\documentclass[preprint,11pt]{elsarticle}

\usepackage[english]{babel}
\usepackage[colorinlistoftodos,bordercolor=orange,backgroundcolor=orange!20,linecolor=orange,textsize=small]{todonotes}
\usepackage{filecontents}
\usepackage[useregional]{datetime2}

\usepackage{fullpage}
\usepackage{caption}
\usepackage{subcaption}
\usepackage{amsmath}
\usepackage{amsthm}
\usepackage{hyperref}
\usepackage{cleveref}
\usepackage{natbib}
\usepackage{centernot}
\usepackage{blkarray}
\usepackage{float}
\usepackage[font=footnotesize,labelfont=bf]{caption}
\usepackage{tikz}

\usetikzlibrary{matrix,shapes,arrows,positioning}
\definecolor{burntorange}{cmyk}{0,0.52,1,0}

\usepackage{stackengine}
\stackMath
\usepackage{booktabs}
\usepackage{matlab-prettifier}

\newcommand{\bigtprod}{\mathop{{\prod}^{}}}
\newcommand{\bigtsum}{\mathop{{\sum}^{\oplus}}}
\newcommand{\tsum}{\sum^{\oplus}}
\newtheorem{theorem}{Theorem}[section]
\newtheorem{lemma}[theorem]{Lemma}
\newtheorem{proposition}[theorem]{Proposition}
\newtheorem{property}[theorem]{Property}
\newtheorem{corollary}[theorem]{Corollary}
\theoremstyle{definition}
\newtheorem{definition}[theorem]{Definition}

\theoremstyle{remark}
\newtheorem{remark}[theorem]{Remark}
\newtheorem{example}[theorem]{Example}

\usepackage{graphicx}

\newcommand{\new}[1]{{\em #1}}

\DeclareMathAlphabet{\mathbbold}{U}{bbold}{m}{n}
\newcommand{\zero}{\mathbbold{0}}
\newcommand{\unit}{\mathbbold{1}}
\newcommand{\zeror}{\mathbbold{0}} 
\newcommand{\mv}{m} 

\newcommand{\rel}{\mathcal{R}}

\newcommand{\R}{\mathbb R}
\newcommand{\vall}{\mathrm v}
\newcommand{\sval}{\mathrm{sv}}

\newcommand{\smax}{\mathbb{S}_{\max}}
\newcommand{\rmax}{\mathbb{R}_{\max}}
\newcommand{\tmax}{\mathbb{T}_{\max}}

\newcommand{\bmaxs}{{\mathbb B}_{{\mathrm s}}}

\newcommand{\PF}{\mathcal{P}_{\!\mathrm{f}}}
\newcommand{\Sp}{\mathfrak{S}}

\newcommand{\X}{\mathsf{X}}

\newcommand{\per}{\mathrm{per}}

\newcommand{\psd}{\operatorname{\mathsf{TPSD}}}
\newcommand{\pd}{{\operatorname{\mathsf{TPD}}}}
\newcommand{\upd}{{\operatorname{\mathsf{UTP}}}}
\newcommand{\ps}{P_{A}}
\newcommand{\tr}{\mathrm{tr}}

\newcommand{\card}{\mathrm{card}}

\newcommand{\elf}{b} 
\newcommand{\balance}{\,\nabla\,}
\newcommand{\notbalance}{\!\centernot{\,\balance}}

\newcommand{\leqsign}{\leq}
\newcommand{\geqsign}{\geq}
\newcommand{\lsign}{<}

\newcommand{\nleqsign}{\not\leqsign}
\newcommand{\gsign}{>}

\newcommand{\adj}{\mathrm{adj}}

\newcommand{\graph}{\mathcal G}

\newcommand{\ext}{\mbox{$\bigwedge$}}
\newcommand{\cycle}{\sigma}

\newcommand{\permutation}{\pi}
\newcommand{\Azero}{\underline{A}}

\usepackage[scr=boondox,scrscaled=1.05]{mathalfa}

\newcommand{\trop}[1][]{\ifthenelse{\equal{#1}{}}{ \mathbb{T} }{ \mathbb{T}(#1) }}

\usepackage{amssymb}

\renewcommand{\geq}{\geqslant}
\renewcommand{\leq}{\leqslant}
\renewcommand{\succeq}{\succcurlyeq}
\renewcommand{\le}{\leq}

\newcommand{\botelt}{\bot}

\newcommand{\N}{\mathbb{N}}
\newcommand{\Z}{\mathbb{Z}}

\DeclareMathOperator{\uval}{ldeg}

\newcommand{\rfield}{\mathcal{L}}

\newcommand{\vgroup}{\Gamma}

\newcommand{\sign}{\mathrm{sgn}}

\begin{document}
\begin{frontmatter}
\title{Spectral Properties of Positive Definite Matrices over Symmetrized Tropical Algebras and Valued Ordered fields}
\author[1]{Marianne Akian\corref{cor}}
\ead{marianne.akian@inria.fr}
\author[1]{Stephane Gaubert}
\ead{stephane.gaubert@inria.fr}
\author[2]{Dariush Kiani\fnref{insf}}
\ead{dkiani@aut.ac.ir}
\author[2]{Hanieh Tavakolipour\fnref{insf}\fnref{postdoc}}
\ead{h.tavakolipour@aut.ac.ir}
\affiliation[1]{organization={Inria and CMAP, Ecole polytechnique, CNRS, Institut Polytechnique de Paris},
country={France}}
\affiliation[2]{organization={Amirkabir University of Technology, Department of Mathematics and Computer Science},
country={Iran}}
\cortext[cor]{Corresponding author}
\fntext[insf]{The study of third and forth authors was funded by Iran National Science Foundation (INSF) (Grant No. 99023636).}
\fntext[postdoc]{This work began when the forth author was a postdoc at Inria and CMAP, Ecole polytechnique, CNRS, Institut Polytechnique de Paris}
\date{\today}


\begin{abstract}
  We investigate the properties of positive definite and positive semi-definite symmetric matrices within the framework of symmetrized tropical algebra, an extension of tropical algebra adapted to ordered valued fields.
  We focus on the eigenvalues and eigenvectors of these matrices. We
 prove that the eigenvalues of a positive (semi)-definite matrix in the tropical symmetrized setting coincide with its diagonal entries. Then, we show that the images by the valuation of the eigenvalues of a positive definite matrix over a valued nonarchimedean ordered field coincide with the eigenvalues of an associated matrix in the symmetrized tropical algebra. Moreover, under a genericity condition, we characterize the images of the eigenvectors under the map keeping track both of the nonarchimedean valuation and sign, showing that they coincide with tropical eigenvectors in the symmetrized algebra.
  These results 
offer new insights into the spectral theory of matrices over tropical semirings, and provide combinatorial formul\ae\ for log-limits of eigenvalues and eigenvectors of parametric families of real positive definite matrices.
\end{abstract}



\setcounter{tocdepth}{3}

\end{frontmatter}

\section{Introduction}

\subsection{Motivation}

Tropical algebra has been introduced by several authors under various names including max-plus algebra or max algebra. It has opened up new avenues in both applied fields, especially in combinatorial optimization and mathematical modelling~\cite{baccelli1992synchronization,butkovivc2010max}, and in pure areas, in particular algebraic geometry~\cite{viro2001dequantization,itenberg2009tropical,maclagan2015introduction}.
The tropical semifield, denoted here $\rmax$, consists of the set $\R\cup\{-\infty\}$ equipped with operations where the maximum of two real numbers replaces addition, and standard addition replaces multiplication.
The absence of a true negation --preventing ordinary term cancellation--
has motivated the introduction of the 
{\em symmetrized} tropical semiring $\smax$ in~\cite{maxplus90b}, as an extension of $\rmax$ in which a formal minus sign serves as a substitute for negation.
This semiring has been employed
as a tool to solve systems of linear equations
and found numerous applications, particularly in the study of matrices, eigenvalues, eigenvectors, and polynomials, and in the tropicalization of linear programming,
see for instance \cite{baccelli1992synchronization,cramer-guterman,benchimol2013,tavakolipour2021,adi}. 
A related construction, called the {\em real tropical hyperfield} or the {\em signed  tropical hyperfield}, was considered in the framework of hyperfields with the aim to study ``real tropicalizations'', i.e., images of real non-archimedean algebraic sets by a map keeping track of valuation and sign, see \cite{viro2010hyperfields,viro2001dequantization}. Recent studies of this hyperfield include \cite{baker2018descartes,Lorsch22,gunn,gunn2}. 
The semiring, $\smax$ with its associated partial order relations,
can be thought of as a hyperfield system as in \cite{Rowen2,AGRowen}.
See~\cite{alessandrini,allamigeon2020tropical,Jell2020} for more information
on ``real tropicalizations''. 

The tropical algebra is intimately related to the notion of valuation over a field. Indeed, a valuation can be seen as a ``morphism'' from a field to the tropical algebra $\rmax$, and to make this morphism property rigorous, one can use the concepts of hyperfields or semiring systems
 (see for instance \cite{baker2018descartes,Rowen2,AGRowen}). 
Valuations are related to asymptotics and 
the role of tropical algebra in asymptotics was recognized by Maslov \cite[Ch. VIII]{maslov1987methodes}, see also \cite{kolokoltsov2013idempotent}, and 
by Viro \cite{viro2001dequantization}.
Valuations are also a way to define notions of tropical 
geometry, see for instance \cite{itenberg2009tropical,maclagan2015introduction}.
Valuations with general ordered groups of values can also be considered,
together with the associated tropical algebra or hyperfield or semiring system.

Similar properties can be obtained by replacing valuations by signed valuations and the tropical semiring by the symmetrized tropical semiring.
Indeed, a signed valuation over an ordered field assigns to any element in the field, the valuation while also indicating their sign.
Then, a signed valuation can be seen as a ``morphism'' from the field to $\smax$. Such a map takes its values in the subset $\smax^\vee$ of signed elements in $\smax$, which also correspond to the elements of the signed tropical hyperfield.
Also, to make this morphism property rigorous, one can use the concepts of hyperfields or semiring systems.
When applied to polynomials, signed valuations reveal the "signed roots"~\cite{gunn,gunn2,tavakolipour2021}.

Positive (semi-)definite symmetric matrices are of particular interest due to their role in various mathematical and engineering applications, such as stability analysis, optimization problems, and systems theory.
Yu characterized the tropicalization of the cone of symmetric positive semidefinite matrices, showing that it is governed by the positivity of $2\times 2$ principal tropical minors~\cite{yu2015tropicalizing}.
This result was extended to the signed case in~\cite{tropicalization},
which also investigates the tropicalization of the cones of completely positive and copositive matrices.
Tropicalizations of other cones relevant to optimization have been studied
in~\cite{Niv_total,Yu2022}. More recently, connections have been identified
between the tropicalization of principal minors of positive semidefinite matrices and discrete concave functions~\cite{yuarxiv2024}.

In classical algebra, the properties of positive definite matrices, particularly their eigenvalues and eigenvectors, are well understood and have been extensively studied. One of the aims of this paper is to introduce and study the eigenvalues and eigenvectors of tropical positive definite matrices in the context of symmetrized tropical algebra.
As a consequence, we will be able to characterize the signed valuations of the eigenvalues and eigenvectors of a positive definite matrix over a real closed field. 

\subsection{Main results}

Our primary contribution is the proof that, in $\smax$, the eigenvalues of a positive (semi)-definite matrix are given by its diagonal entries, see~\Cref{sec:eig}. This result 
offers practical computational advantages, as it simplifies the determination of eigenvalues in symmetrized tropical setting. We build upon the results presented in \cite{tavakolipour2021} and specially \Cref{coro2-uniquefact} to demonstrate that the characteristic polynomial of a positive definite matrix over $\smax$ admits a unique factorization. This result helps us to define the multiplicity of the eigenvalues of such a matrix and to show that the multiplicity of 
any eigenvalue coincides with the number of its occurrences as a diagonal element of the matrix.

Some notions of generalized eigenvectors associated to the eigenvalues over $\rmax$ have already been investigated in the literature in particular in the work of Izhakian and Rowen~\cite{izhakianmatrix3} and in the works of Nishida and co-authors, see \cite{Nishida2020,Nishida2021,nishida2021independence}. In \Cref{eig_vec}, we define a (generalized) notion of geometric eigenvalue and eigenvector over $\smax$. Moreover in \Cref{smaxeigenvector-ws}, we introduce refined concepts of 
weak and strong eigenvectors, respectively.
This offers more tools for analyzing the algebraic structure of matrices, 
and allows us in some cases to determine eigenvectors using the adjoint matrix (see \Cref{spec-eig-vector}).

Using these tools, we identify candidates for all the eigenvectors of a positive definite matrix over $\smax$ (see \Cref{coro-unique-eigen}). Furthermore in \Cref{subsec:kleen}, we characterize these candidate eigenvectors using the Kleene star operation. Such a characterization 
may be thought as a generalization of the notion of eigenvector over $\rmax$
introduced in \cite{Nishida2020}.
Then, in \Cref{sec-generic}, we show that generically these candidate eigenvectors are the unique eigenvectors. 

Finally, in \Cref{sec:apps}, 
we show that generically, the signed valuations of the eigenvalues and eigenvectors of a positive definite matrix over a real closed field coincide with the signed tropical eigenvalues and eigenvectors of the signed valuation of the matrix.

This can be compared to a characterization of the asymptotic behavior of the eigenvalues and eigenvectors of a parametric family of positive definite matrices over an ordered field, using the eigenvalues and eigenvectors of a positive definite matrix over $\smax$.
This result provides 
new insights into the nature of eigenvalues  and eigenvectors of usual positive (semi-)definite matrices.
We also show a Gershgorin type bound for the eigenvalues of a positive definite real matrix.

\bigskip

The paper is structured as follows. We begin with a review in \Cref{sec-elem} of the basic principles of tropical and  symmetrized tropical algebra, and in \Cref{sec-matpol} of the definitions and known or elementary properties of the algebraic constructions within these frameworks, such as matrices, polynomials, eigenvalues and eigenvectors. We then explore in \Cref{sec:3} the concepts of positive (semi)-definite matrices over $\smax$, detailing the theoretical developments and methods used to derive our results. In particular, we characterize the eigenvalues of these matrices over $\smax$. In \Cref{sec:3p}, we give several characterizations of the eigenvectors of these matrices over $\smax$. 
Finally, in \Cref{sec:apps}, we examine the relationship between the eigenvalues of matrices over ordered fields and their counterparts in symmetrized tropical algebra. We finish by illustrating the results by some numerical results on the eigenvalues and eigenvectors of parametric families of positive definite matrices. 

\section{Definitions and elementary properties}\label{sec-elem}
In this section, we review some necessarily definitions, notations and results of max-plus or tropical and symmetrized max-plus or tropical algebra. See~\cite{baccelli1992synchronization, butkovivc2010max,maclagan2015introduction} for more information. 

\subsection{Preliminaries of max-plus or tropical algebra $\rmax$ and $\tmax$} 
\begin{definition}
Let $\R$ be the set of real numbers. The tropical semiring, $\rmax$, is the set $\R \cup \{-\infty\}$ 
equipped with
  the addition $(a,b)\mapsto a\oplus b:=\max\{a,b\}$, with the zero element $\zero:=-\infty$ and
 the multiplication  $(a,b)\mapsto a\odot b:=a+b$, with the unit element $\unit:=0$.
\end{definition}

\begin{example}
 Over $\rmax$, we have
\begin{itemize}
\item $1 \oplus -2 = 1$
 \item $6 \odot 2 = 8$
 \item $2^{ 3}= 2\odot  2\odot 2= 6$.
 \end{itemize}
 \end{example}
We shall also use the more general family of tropical semifields defined 
as follows, see also \cite{tavakolipour2021}.
\begin{definition} \label{tmax}
Given a (totally) ordered abelian group $(\vgroup,+,0,\leq)$,
we consider an element $\botelt$ satisfying
$\botelt \leq a$ for all $a\in\vgroup$, and  which 
does not belong to $\vgroup$.
Then, the {\em tropical semifield} over $\vgroup$, denoted 
 $\tmax(\vgroup)$, is the set $\vgroup \cup\{\botelt\}$,
equipped with 
the addition 
$(a,b) \mapsto a\oplus b:= \max(a,b)$, with zero element $\zero:=\botelt$,
and multiplication $(a,b)\mapsto a\odot b:= a+b$, and 
 $\botelt \odot a=a \odot\botelt= \botelt$,  for all $a,b\in \vgroup$,
so with unit $\unit:=0$.
\end{definition}
In particular, the zero element $\botelt$ is absorbing.
The $n$-th power of an element $a\in\vgroup$ for the multiplicative law $\odot$,
$a^n:=a \odot \ldots \odot a$  ($n$-times), coincides with the
sum  $a+ \dots + a$ ($n$-times), also denoted by $na$.
We say that the group $\vgroup$ is {\em divisible},
if for all $a\in \vgroup$ and for all positive integers $n$,
there exists $b$ such that $nb=a$. 
In this case, $b$ is unique (since $\vgroup$ is ordered).
We say that $\vgroup$ is {\em trivial} if it is equal to $\{0\}$.
When $\vgroup=\R$, we recover $\rmax$. 

 \subsection{Preliminaries of symmetrized max-plus algebra $\smax$}

Here we recall the construction and basic properties of the symmetrized
 tropical semiring. We refer the reader to \cite{baccelli1992synchronization,gaubert1992theorie,cramer-guterman} for information at a more detailed level
in the case where $\vgroup=\R$. We describe here the generalization to the case
of any ordered group $\vgroup$, which was presented in \cite{tavakolipour2021}.

Let us consider the set $\tmax^2:=\tmax\times \tmax$ endowed with operations $\oplus$ and $\odot$:
\[(a_1,a_2) \oplus (b_1,b_2) =(a_1\oplus b_1, a_2 \oplus b_2),\]
\[(a_1,a_2) \odot (b_1,b_2) = (a_1 b_1 \oplus a_2  b_2, a_1  b_2 \oplus a_2  b_1),\]

with $\zero:=(\botelt,\botelt)$ as the zero element and $\unit:=(0, \botelt)$ as the unit element. 
Define the following three operators on  $a= (a_1, a_2)\in \tmax^2$:
\begin{center}
\begin{tabular}{ll}
$\ominus a = (a_2, a_1)$ & minus operator $\tmax^2\to \tmax^2$;\\ 
$|a| = a_1 \oplus a_2$ & absolute value $\tmax^2\to \tmax$;\\  
$a^{\circ} = a\ominus a = (|a|, |a|)$&  balance operator $\tmax^2\to \tmax^2$.
\end{tabular}
\end{center}
The operator $\ominus$ satisfies all the properties of a minus sign except that
$a\ominus a$ is not zero except when $a=\zero$.
We also define the \new{balance relation} over 
$\tmax^2$ as follows:
\[  (a_1, a_2) \balance  (b_1, b_2) \Leftrightarrow a_1 \oplus b_2 = a_2 \oplus b_1\enspace .\]
It satisfies 
\begin{equation}
a \balance b \Leftrightarrow a \ominus b\balance \zero\enspace .\end{equation}
Balance relation is reflexive, symmetric, and compatible with
addition and multiplication of $\tmax^2$. However, it is not an equivalence relation, because it lacks the expected transitive property. For example (for $\vgroup=\R$),  we have $(1,2) \balance (3,3)$, $(3,3) \balance (1,1)$, but $(1,2)\notbalance(1,1)$. We then consider the following relation $\mathcal{R}$ on $\tmax^2$ which refines the balance relation: 
\[(a_1,a_2) \mathcal{R} (b_1,b_2) \Leftrightarrow
\begin{cases}
a_1 \oplus b_2 = a_2 \oplus b_1& \;\text{if}\; a_1 \neq a_2, \;b_1 \neq b_2,\\
(a_1,a_2)=(b_1,b_2)& \text{otherwise.}
\end{cases}
\]  

\begin{example}
To better understanding the difference between $\balance$ and $\rel$, in the following table we compare them for few examples (with $\vgroup=\R$).
\[\begin{array}{c|cccc}
&(1,4)&(4,1)&(4,4)&(3,3)\\
\hline
(1,4)&\balance,\rel&\notbalance,  \centernot\rel& \balance,\centernot\rel&\notbalance,  \centernot\rel\\
(4,1)&\notbalance,  \centernot\rel&\balance,\rel&\balance,\centernot\rel&\notbalance,  \centernot\rel\\
(4,4)&\balance,  \centernot\rel&\balance,  \centernot\rel&\balance, \rel&\balance,  \centernot\rel\\
(3,3)&\notbalance,  \centernot\rel&\notbalance,  \centernot\rel&\balance, \centernot\rel&\balance,  \rel
\end{array}\]
\end{example}

One can check that $\mathcal{R}$ has the transitive property and so is an equivalence relation on $\tmax^2$. Also it is compatible with $\oplus$ and $\odot$ of $\tmax^2$, $\balance$, $\ominus$, $|.|$ and $^{\circ}$ operators,
which then can be defined on the quotient $\tmax^2 / \mathcal{R}$.
\begin{definition}[$\smax$]\label{def:sym_def}
The \new{symmetrized tropical semiring} is the quotient semiring $(\tmax^2 / \mathcal{R},\oplus,\odot)$ and is denoted by $\smax$ or $\smax(\vgroup)$. 
We denote by $\zero:=\overline{(\botelt, \botelt)}$ the zero element 
and by $\unit:=\overline{(0, \botelt )}$ the unit element.
We also use the notation  $ab$ for  $a\odot b$ with  $a,b\in\smax$, 
and $a^n$ for the product $a\odot \cdot \odot a$ n-times.
\end{definition}\label{def:smax}
We distinguish three kinds of equivalence classes (\cite{gaubert1992theorie}):
\begin{center}
\begin{tabular}{ll}
$\overline{(c, \botelt)} = \{(c,a_2)\mid a_2<c\}, \; c\in \vgroup$ & positive elements  \\ 
$\overline{(\botelt,c)}=\{(a_1, c)\mid a_1<c\}, \; c\in \vgroup$ & negative elements  \\  
$\overline{(c,c)}=\{(c,c)\}, \; c\in \vgroup\cup\{\botelt\}$ & balance elements.
\end{tabular}
\end{center}
Then, we denote by $\smax^{\oplus}$,
$\smax^{\ominus}$ and  $\smax^{\circ}$ 
the set of positive or zero elements,
the set of negative or zero elements, and the set of balance elements,
respectively.
Therefore, we have:
\[\smax^{\oplus}\cup \smax^{\ominus}\cup \smax^{\circ}=\smax, \]
where the pairwise intersection of any two of these three sets
is reduced to $\{\zero\}$.

\begin{property}
The subsemiring $\smax^{\oplus} $ of $\smax$ can be
identified to $\tmax$, by the morphism $c\mapsto \overline{(c, \botelt)}$.
This allows one to write $a \ominus b$ instead of $\overline{(a, \botelt)} \oplus \overline{(\botelt,b)}$.
\end{property}


\begin{property}\label{prop-modulus}
Using the above identification, the absolute value map $a\in \smax \mapsto |a|\in \smax^\oplus$ is a morphism of semirings.
\end{property}


\begin{definition}[Signed tropical elements]\label{signed_elements}
The elements of $\smax^\vee:=\smax^{\oplus} \cup \smax^{\ominus}$ are called \new{signed tropical elements}, or simply \new{signed elements}. They are either positive, negative or zero.
\end{definition}


\begin{remark}
The elements of $\smax^{\circ}$ play the role of the usual zero element. 
Moreover, 
the set $\smax \setminus \smax^{\circ}=\smax^\vee\setminus\{\zero\}$ is the set of all invertible elements of $\smax$.
\end{remark}

\subsection{Relations over $\smax$}
\begin{definition}\label{partial_order}
We define the following relations, for $a,b \in \smax$:
\begin{enumerate}
\item $a \preceq b \iff b = a \oplus c \;\text{for some}\;c \in \smax \iff b=a\oplus b$ ;
\item $a \prec b \iff a \preceq b, \; a \neq b$ ;
\item $a \preceq^{\circ} b \iff b = a \oplus c \;\text{for some}\;c \in \smax^{\circ}$.
\end{enumerate}

\end{definition}

The relations $\preceq$ and $\preceq^\circ$ in \Cref{partial_order} are partial orders (they are reflexive, transitive and antisymmetric).

\begin{example} We have the following inequalities: 
\begin{enumerate}
\item
$\zero \preceq \ominus 2 \preceq \ominus 3,\;\zero \preceq 2 \preceq 3,\; 2 \preceq \ominus 3$ ;
\item $3$ and $\ominus 3$ are not comparable for $\preceq$ ;
\item  $1\preceq^{\circ} 2^{\circ}$,\;$\ominus 1\preceq^{\circ}  2^{\circ}$,\; $\ominus 2 \preceq^{\circ}  2^{\circ}$ ;
\item $3$ and $2^{\circ}$  are not comparable for $\preceq^{\circ}$. 
\end{enumerate}
\end{example}

\begin{property}\label{property-preceq}Let $a,b \in \smax$.
\begin{enumerate}
\item If $|a| \prec |b|$, then $a \oplus b = b$.
\item If $a \preceq b$, $|a|=|b|$ and $b \in \smax^{\vee}$, then $a=b$.
\item If $b \in \smax^{\vee}$, then $a \preceq^{\circ} b $ iff $a=b$.
\item If $|a| \preceq |b|$ and $b \in \smax^{\circ}$, then $a \preceq^{\circ} b $ and so $a \preceq b$.
\item  $a \oplus b =b \Rightarrow |a| \preceq |b|$.
\end{enumerate}
\end{property}

In \cite{tropicalization}, the authors equiped $\smax$ with other 
``order'' relations, by using a relation on $\tmax^2$ and then quotienting, and used them to define positive semidefinite matrices over $\smax$. We give the definition directly on $\smax$ in \Cref{partial_order2} below, while replacing the notations $\preceq$ and $\prec$ of
\cite{tropicalization} by the notations $\leqsign$ and $\lsign$,
since we already used the notation $\preceq$ for the natural order of  $\smax$.
\begin{definition}\label{partial_order2}\cite{tropicalization}\
For $a,b \in \smax$:
\begin{enumerate}
\item $a  \leqsign  b  \iff b \ominus a \in  \smax^{\oplus}\cup \smax^{\circ}$ ;
\item $a  \lsign  b  \iff b \ominus a \in  \smax^{\oplus}\setminus\{\zero\}$. 
\end{enumerate}
\end{definition}
\begin{example}
Using the relations in \Cref{partial_order2}, we have the following properties:
\begin{enumerate}
\item $\ominus 3 \lsign \ominus 2 \lsign \zero \lsign    2 \lsign  3$\enspace; 
\item $\leqsign$ is not antisymmetric on $\smax$: 
$2 \leqsign 3^{\circ}$ and $3^{\circ} \leqsign 2$\enspace;
\item $\leqsign$ is not transitive on $\smax$: 
$2 \leqsign 3^{\circ}, 3^{\circ} \leqsign 1$ but 
$2  \nleqsign  1$\enspace.
\end{enumerate}
\end{example}
The relation $\leqsign$ is reflexive, but it is not antisymmetric, nor transitive on $\smax$,  as shown in the examples above. 
However, on $\smax^{\vee}$,  $\leqsign$ is a total order 
and $\lsign$ coincides with ``$\leqsign$ and $\neq$'',
 see \Cref{order_new} and \Cref{order-exp} below.

\begin{proposition}\cite{tropicalization}\label{order_new}
Let $a, b , c \in \smax$.
\begin{enumerate}
\item $a \leqsign   a$ for any $a \in \smax$ ($\leqsign  $ is reflexive);
\item $a \leqsign   b$ and $b \leqsign   a$ if and only if $a \balance b$; hence $\leqsign  $ is antisymmetric on $\smax^{\vee}$;
\item If $a \leqsign   b$ and $b \leqsign   c$ and $b \in \smax^{\vee}$ then $a \leqsign   c$; hence $\leqsign  $ is transitive on $\smax^{\vee}$.
\end{enumerate}
\end{proposition}
\begin{property}\label{order-exp}
If we identify the elements of  $\smax^\vee$ with elements of $\R$ by 
the map $\ominus a\mapsto -\exp(a)$, $\oplus a\mapsto \exp(a)$ and $\zero\mapsto 0$, then, we get that the relations $ \leqsign $ and $\lsign$ on $\smax^\vee$ are the usual order $\leq$  and the strict relation $<$ on $\R$.
Moreover, on $\smax^\oplus$, the relations $ \leqsign $ and $\lsign$
are equivalent to the relations $\preceq$ and $\prec$, and to 
the usual order and its strict version on the set $\tmax$.
\end{property}
We have also the following properties, which can be deduced easily from 
\Cref{partial_order2}.
\begin{lemma}\label{product_order}
Let $a, b, c\in \smax^{\vee}$. Then we have  
\begin{enumerate}
\item $a \leqsign  b, \;c \geqsign \zero \Rightarrow a c \leqsign  b c\enspace,$
\item $a \lsign  b, \;c \gsign \zero \Rightarrow a c \lsign  b c\enspace.$ \hfill \qed
\end{enumerate}
\end{lemma}
\begin{lemma}\label{modulus_order}
Let $a, b\in \smax^{\vee}$. Then $a^{ 2} \lsign b^{ 2}$ if and only if $|a| \lsign |b|$. Similarly, $a^{ 2} \leqsign b^{ 2}$ if and only if $|a| \leqsign |b|$. 
\end{lemma}
\begin{proof}
Any $a\in \smax^{\vee}$ can be written as $a=|a|$ or $a=\ominus |a|$,
using the above identifications.
So $a^{ 2}=|a|^{ 2}$, where $|a|\in \smax^\oplus$.
Then, we only need to check the equivalences of the lemma 
for $a,b\in \smax^\oplus$. Since in $\smax^\oplus$,
$\lsign$ and $\leqsign$ are equivalent to $\prec$ and $\preceq$, respectively,
or to the usual order and its strict version on $\tmax$, we obtain the
result of the lemma.
\end{proof}





\begin{property} \label{equality_balance}
The relation $\balance$ satisfies the following properties, for $a,b \in \smax$:
\begin{enumerate}
\item\label{pro1} We have $a \balance b \Leftrightarrow a \ominus b\balance \zero$.
\item If $a,b \in \smax^{\vee}$ and $a \balance b$, then we have $a=b$.
\item If $b \in \smax^{\vee}$, $a \balance b$ and $a\preceq b$, 
then we have  $a=b$.
\end{enumerate}
\end{property}

\section{Preliminaries on matrices and polynomials over $\smax$}\label{sec-matpol}
\subsection{Matrices} 
Given any semiring $(\mathcal{S},\oplus,\zero,\odot,\unit)$ (such as $\rmax$, $\tmax$  or $\smax$), we denote by $\mathcal{S}^{n}$ and $\mathcal{S}^{n\times m}$ the sets of $n$-dimensional vectors and of $n\times m$ matrices with entries in $\mathcal{S}$. We also use the notation $ab$ for  $a\odot b$ with  $a,b\in \mathcal{S}$, and $a^n$ for the product $a\odot \cdot \odot a$ n-times. Then, the finite sum $\tsum$ and product $\prod$ notations, and the matrix multiplication, addition and power operations over $\mathcal{S}$ are defined as in usual linear algebra.
For example if $A=(a_{ij}) \in \mathcal{S}^{n\times m}$ and $B=(b_{ij}) \in \mathcal{S}^{m\times p}$,
then $A B\in \mathcal{S}^{n\times p}$  and has entries $(A  B)_{ij}=\tsum_k a_{ik} b_{kj}$. Also, for any $n\geq 1$, we denote by $\zero$, and call the zero vector, the $n$-dimensional vector with all entries equal to $\zero$, and by $I$, the $n\times n$ identity matrix over $\mathcal{S}$ with diagonal entries equal to $\unit$ and off-diagonal entries equal to $\zero$. Finally, for a square $n\times n$ matrix $A$, we denote $A^{ 2}=A A$, etc, with $ A^{ 0}$ equal to the identity matrix $I$.

For any positive integer $n$, denote by $[n]$ the set $\{1, \ldots, n\}$. 
We denote by  $\Sp_{n}$, the set of all permutations of $[n]$.
Recall that a \new{cycle} in $[n]$ is a sequence  
$\cycle=(i_{1},i_{2},\ldots , i_{k})$ of different elements of $[n]$,
with the convention that $i_{k+1}=i_1$, and that any
permutation $\permutation$ of $[n]$ can be decomposed uniquely
into disjoint cycles which cover $[n]$,
meaning that $\permutation(i_\ell)= i_{\ell+1}$ for all $\ell\in [k]$ and all
cycles $\cycle=(i_{1},i_{2},\ldots , i_{k})$ of $\permutation$.

Let $A =(a_{ij}) \in \mathcal{S}^{n \times n}$ be a matrix.
For any permutation $\permutation$ of $[n]$, the weight of $\permutation$ associated to $A$
is given by
\[ w(\permutation)=\bigtprod_{i \in[n]}a_{i\permutation(i)}\enspace ,\]
and the weight of any cycle $\cycle=(i_{1},i_{2},\ldots , i_{k})$ associated to $A$ is given by
\[w(\cycle)=\bigtprod_{\ell\in [k]} a_{i_\ell i_{\ell+1}}\enspace .\]
Then, as in usual algebra, the weight of a permutation  is the product of the weights of
its cycles.



\begin{definition} \label{per}The \new{permanent} of a matrix $A=(a_{ij}) \in \mathcal{S}^{n \times n}$  is
\[\per(A)= \bigtsum_{\permutation \in \Sp_{n}} \bigtprod_{i \in[n]}a_{i\permutation(i)}
=\bigtsum_{\permutation \in \Sp_{n}} w(\permutation)
\enspace .
\]
\end{definition}
When the semiring $\mathcal{S}$ has a negation map, we can also define the determinant. We only give the definition in $\smax$.
\begin{definition}[Determinant]\label{det_s}
Let $A=(a_{ij})$ be an  $n \times n$ matrix over $\smax$. The \new{determinant} is 
\[\det(A):= \bigtsum_{\permutation \in \Sp_n} \mathrm{sgn}(\permutation) \bigtprod_{i\in [n]} a_{i\permutation(i)} = \bigtsum_{\permutation \in \Sp_n} \mathrm{sgn}(\permutation) w(\permutation)
\enspace ,\]
where \[\mathrm{sgn}(\permutation)=\begin{cases}
\unit & \;\text{if}\;\permutation \;\text{is even};\\
\ominus \unit & \text{otherwise}.
\end{cases}\]
\end{definition}
This allows one to define also the adjugate matrix.
\begin{definition}[Adjugate]\label{def-adjugate}
The adjugate matrix of $A=(a_{ij}) \in \smax^{n \times n}$ is the matrix $A^{\mathrm{adj}}\in \smax^{n\times n}$ with entries:
\[ (A^{\mathrm{adj}})_{i,j} := (\ominus 1)^{i+j} \det(A[\hat{j},\hat{i}])\enspace , \]
where $A[\hat{j},\hat{i}]$ is the matrix obtained after eliminating the $j$-th row and the $i$-th column of $A$.
\end{definition}

For any matrix $A$ with entries in $\smax$, we denote by $|A|$ the
matrix with entries in $\tmax$ obtained by applying the modulus map
$|\cdot|$  entrywise.
\begin{remark}\label{perdet}
For $A \in (\smax)^{n \times n}$, we have
$|\det(A)|=\per(|A|)$.
\end{remark}

\begin{lemma}[\protect{\cite{akian2009linear}}]\label{adj} 
Let $A \in (\smax^\vee)^{n \times n}$. Then the following balance relation holds
\[A A^{\mathrm{adj}} \succeq^{\circ}  \det(A) I .\]
In particular if $\det(A) \balance \zero$ then $A A^{\mathrm{adj}} \balance \zero$.
\end{lemma}
We now recall some results about the solution of linear systems over $\smax$.

\begin{theorem}[\cite{maxplus90b,cramer-guterman}]\label{cramer}
Let $A \in (\smax)^{n \times n}$ and $b \in (\smax)^{n}$, then
\begin{itemize}
\item every solution $x \in (\smax^{\vee})^{n}$ of the linear system $A x \balance b$ satisfies the relation 
\begin{equation}\label{cram}\det(A) x \balance A^{\adj}  b\enspace.
\end{equation}
\item If $A^{\adj} b \in (\smax^{\vee})^{n}$ and $\det(A)$ is invertible, then 
\[\tilde{x} = \det(A)^{ -1} A^{\adj}  b\]
is the unique solution of $A x \balance b$ in $(\smax^{\vee})^{n}$.
\end{itemize}
\end{theorem}
\begin{remark}\label{ith_cramer}
Let $D_{x_i}$, be the determinant of the matrix obtained by replacing the $i$-th column of $A$ with $b$. Then $(A^{\adj}b)_i=D_{x_i}$. When $\det(A)$ is invertible, \Cref{cram} is equivalent to 
$(\forall i) \;x_i \balance \det(A)^{-1}D_{x_i}$,
where the right hand side of this equation is exactly the classical $i$-th Cramer formula.
\end{remark}

\begin{theorem}[\cite{maxplus90b,cramer-guterman}]\label{existence_signed}
Let $A \in (\smax)^{n \times n}$. Assume that $\det(A)\neq \zero$ (but possibly $\det(A) \balance \zero$). Then for every $b \in (\smax)^{n}$ there exists a solution $x \in (\smax^{\vee})^n$ of $A x \balance b$, which can be chosen in such a way that $|x|=|\det(A)|^{ -1} |A^{\adj} b|$.
\end{theorem}
\begin{theorem}[Homogeneous systems over $\smax$ \protect{\cite[Th. 6.5]{maxplus90b}, see also \cite[Th. 6.1]{cramer-guterman}}]\label{homo}
Let $A \in (\smax)^{n \times n}$, then there exists a solution $x \in (\smax^{\vee})^{n}\setminus\{\zero\}$ to the linear system $A x \balance \zero$ if and only if $\det(A)\balance \zero$. 
\end{theorem}

We shall also use the following construction.
The semirings $\rmax$, $\tmax$, and  $\smax$ are all topological semirings
(meaning that operations are compatible with the topology), when 
endowed with the topology of the order $\leq$ for $\tmax$ and $\preceq$ for 
$\smax$.  They are also idempotent meaning that $a\oplus a=a$ for all $a$,
so that the sum of elements is also the supremum.
They are also relatively complete for their associated partial order, meaning that the 
supremum of an upper bounded set always exists, or that they become complete
when adding a top element to them.
In what follows, $\mathcal{S}$ will be $\rmax$, $\tmax$, and  $\smax$,
but it can be any idempotent semiring which is relatively complete for the associated partial order (such that $a\leq b$ if $a\oplus b=b$).
  \begin{definition}(Kleene's star)\label{star_smax}
The Kleene's star of a matrix $A \in \mathcal{S}^{n \times n}$, denoted $A^*$, is defined as the sum $\tsum_{k\geq 0}A^{ k}$, if the series converges to a matrix over $\mathcal{S}$. Recall that $ A^{ 0}=I$ the identity matrix.
\end{definition}

To any matrix $A =(a_{ij}) \in \mathcal{S}^{n \times n}$,  we associate the weighted directed graph $\graph(A)$ with set of nodes $[n]$, set of edges 
$E=\big\{(i,j): a_{ij}\neq \zero,\; i,j \in [n]\big\}$, 
and in which the weight of an edge $(i,j)$ is $a_{ij}$.
Then, a path 
in $\graph(A)$ of  length $k\geq 1$ is a sequence  
$(i_1, \ldots, i_{k+1})$ such that $(i_\ell,i_{\ell+1})\in E$, for
all $\ell\in [k]$, 
it has initial node $i_1$, final node $i_{k+1}$, and 
weight 
$\bigtprod_{\ell\in [k]} a_{i_\ell i_{\ell+1}}$. 
By convention, a path of length $0$ has weight $\unit$ and 
its initial and final nodes are equal.
We say that the matrix $A$ is irreducible if $\graph(A)$ is strongly connected, meaning that there is a path from each node to another node.

\begin{property}\label{irreducible}
 Let $A =(a_{ij}) \in \mathcal{S}^{n \times n}$ be 
such that $A^*$ exists.
Then, for all $i,j\in [n]$, the entry $A^*_{ij}$ 
is equal to the supremum of the weights of all paths with initial node $i$ and
final node $j$.

If $A$ is irreducible, then, $A^*$ has no zero entries.
\end{property}

\subsection{Polynomials over $\rmax$, $\tmax$  and $\smax$}
\label{sec-polynomials}
The following definitions are the same as in usual algebra.
\begin{definition}[Formal polynomial] Given any semiring $(\mathcal{S},\oplus,\zero,\odot,\unit)$ (such as $\rmax$, $\tmax$  or $\smax$), 
a (univariate) \new{formal polynomial} $P$ over $\smax$ is a sequence $(P_k)_{k\in \mathbb{N}} \in \mathcal{S}$, where $\mathbb{N} $ is the set of natural numbers (including $0$),  such that $P_k=\zero$ for all but finitely many values of $k$. We denote a formal polynomial $P$ as a formal sum, $P = \tsum_{k\in \mathbb{N}} P_{k}  \X^{k}$, and the set of formal polynomials as $\mathcal{S}[\X]$.
This set is endowed with the following two internal operations, which make it 
a semiring:
coefficient-wise wise sum, $(P \oplus Q)_k=P_k \oplus Q_k$; and 
Cauchy product, $(P Q)_k= \tsum_{0 \leq i \leq k}P_i Q_{k-i}$.

A formal polynomial reduced to a sequence of one element is called a \new{monomial}. 
\end{definition}

When the semiring $\mathcal{S}$ is $\smax$, we apply the absolute value map $|\cdot|$, the balance relation $\balance$, and the relations of \Cref{partial_order} and \Cref{partial_order2} to formal polynomials coefficient-wise.

\begin{example}
 $P=\X^4 \oplus   \unit^{\circ}\X^{3} \oplus  \unit^{\circ}\X^2 \oplus \unit^{\circ} \X \ominus \unit $
and
 $Q= \X^4 \ominus \unit$,
are two examples of formal polynomials over $\smax$,
and we have $Q\preceq^\circ P$ and $Q\lsign P$.
\end{example}

\begin{definition}[Degree, lower degree and support]
 The \new{degree} of $P$ is defined as
\begin{equation}\label{deg}\deg(P):=\sup\{k \in \mathbb{N} \mid P_k \neq \zeror\},\end{equation}
and \new{lower degree} of $P$ is defined as 
\begin{equation}\label{valuation}\uval (P) := \inf\{k \in \mathbb{N}\;|\;P_k \neq \zeror\}.\end{equation}
In the case where $P = \zeror$, we have $\deg(P)=0$ and $\uval(P) = +\infty$. 

We also define the \new{support} of $P$ as the set of indices of the non-zero elements of $P$, that is
$\mathrm{supp}(P):=\{k\in \mathbb{N} \mid P_k \neq \zeror\}$.
We say that a formal polynomial has a \new{full support} if 
$P_k\neq \zeror$ for all $k$ such that $\uval(P) \leq k \leq \deg(P)$.
\end{definition}
\begin{definition}[Polynomial function]
To any $P \in \mathcal{S}[\X]$, with degree $n$ and lower degree $\mv$,
we associate a \new{polynomial function} 
\begin{equation}\label{widehat_p}\widehat{P}: \mathcal{S} \rightarrow \mathcal{S} \; ; \; x \mapsto \widehat{P}(x)= \bigtsum_{\mv\leq k\leq n}P_{k} x^{ k}.\end{equation} 
We denote by $\PF(\smax)$, the set of polynomial functions $\widehat{P}$.

 \end{definition}
We now consider the special case where $\mathcal{S}$ is $\rmax$, $\tmax$ or $\smax$ semiring.
From now on, we shall assume that $\vgroup$ is {\bf divisible}.

\subsubsection{Roots of polynomials over $\rmax$ and $\tmax$}
When the semiring $\mathcal{S}$ is $\rmax$ or $\tmax$, the addition in \eqref{widehat_p}
is the maximization. Roots of a polynomial are defined as follows.
\begin{definition}[$\rmax$ and $\tmax$-roots and their multiplicities] \label{def_corners}
Given a formal polynomial $P$ over $\rmax$ (resp.\ $\tmax$),
and its associated polynomial function $\widehat{P}$,
 the non-zero $\rmax$ (resp.\ $\tmax$)-\new{roots} of $P$ or $\widehat{P}$ 
 are the points $x$ at which the maximum 
in the definition \eqref{widehat_p} of 
$\widehat{P}$ as a supremum of monomial functions,
is attained at least twice (i.e.\ by at least two different monomials).
Then, the multiplicity of $x$ is the difference between the largest and the smallest exponent of the monomials of $P$ which attain the maximum at $x$.

 If $P$ has no constant term, then $\zero$ is also a $\rmax$ (resp.\ $\tmax$)-root of $P$, and its multiplicity is equal to the lower degree of $P$.
 \end{definition}
Non-zero $\rmax$-roots of a formal polynomial $P$
are also the points of non-differentiability of 
$\widehat{P}$, and their multiplicity is
also the change of slope of the graph of  $\widehat{P}$ at these points.



 The following theorem states the fundamental theorem of tropical algebra  which was shown by Cuninghame--Green and Meijer for $\rmax$ and stated in
\cite{tavakolipour2021} for $\tmax$.

\begin{theorem}[\cite{cuninghame1980algebra} for $\rmax$]
Every formal polynomial  $P \in \rmax[\X]$ (resp.\ $\tmax[\X]$) of degree $n$ has exactly $n$ roots $c_1\geq \cdots \geq c_n$ counted with multiplicities, and the associated polynomial function $\widehat{P}$ can be factored in a unique way as 
\[\widehat{P}(x)=  P_n (x \oplus c_1)  \cdots  (x \oplus c_n)
\enspace. \]
\end{theorem}


The following result was shown for $\rmax$ in \cite{baccelli1992synchronization} and stated for $\tmax$ in \cite{tavakolipour2021}.
\begin{lemma}[See~\protect{\cite[p.\ 123]{baccelli1992synchronization}} for $\vgroup=\R$]\label{roots_poly} Consider a formal polynomial $P$  over $\rmax$ (resp.\ $\tmax$) of lower degree    $\mv$ and degree $n$.
\begin{itemize}
\item If $P$ is of the form $P=P_n (\X \oplus c_1)\cdots (\X \oplus c_n)$ (where $c_i$ maybe equal to $\zeror$), then $P$ has full support and satisfies:
\begin{equation}
\label{concavepoly}
P_{n-1}-P_n \geq P_{n-2}-P_{n-1} \geq \cdots \geq P_{\mv}-P_{\mv +1}.\end{equation}
\item
Conversely, if $P$ satisfies \eqref{concavepoly}, then
$P$ has full support, the numbers $c_i \in \rmax$ defined by 
\[c_i := \begin{cases}
P_{n-i} - P_{n-i+1}& 1 \leq i \leq n-\mv;\\
\zeror & n-\mv <i \leq n.
\end{cases}
\]
are such that 
$c_1  \geq \cdots \geq c_n$
and $P$ can be factored as 
$P=P_n (\X \oplus c_1)\cdots (\X \oplus c_n)$.
\end{itemize}
If $P$ satisfies one of the above conditions, we shall say that
$P$ is {\em factored}.
\end{lemma}
Over $\rmax$, the condition \eqref{concavepoly} means that the coefficient map from $\N$ to $\R\cup\{-\infty\}$ is concave.
   

\subsubsection{Roots of polynomials over $\smax$}


Let us denote by $\smax^\vee[\X]$ the subset of
$\smax[\X]$ of formal polynomials over $\smax$ with coefficients in
$\smax^\vee$. In \cite{tavakolipour2021},
 we only considered roots of such polynomials
and their multiplicities. 
Since characteristic polynomials of matrices need not have coefficients 
in $\smax^\vee$, one may need to generalize these notions. 
For this purpose, we shall consider below a notion equivalent to the notion of ``corner root'' introduced in \cite[Section 6]{adi}  for a general semiring with a symmetry and a modulus,
which is then used  to define eigenvalues of matrices,
and which applies in particular to the case of $\smax$ semiring.

\begin{definition}[$\smax$ or $\smax^\vee$-roots and factorization]
\label{def-smaxroots}
Suppose that $P\in \smax[\X]$. Define 
$P^{\vee}$ as the element of $\smax^{\vee}[\X]$ such that 
for all $i\in \N$, 
$P^{\vee}_i=P_i$ if $P_i\in \smax^{\vee}$ and $P^{\vee}_i=\zero$ otherwise.
Then, 
the $\smax$-\new{roots} (resp.\  $\smax^{\vee}$-\new{roots}) of $P$   are the signed elements $r \in \smax^{\vee}$ for which $\widehat{P}(r) \balance  \zero$
(resp.\ $\widehat{P}(r)=\widehat{P^{\vee}}(r) \balance  \zero$).
When $P\in\smax^{\vee}[\X]$, $\smax^\vee$-\new{roots} of $\widehat{P}$ are defined as 
$\smax$-roots or equivalently $\smax^{\vee}$-roots of $P$.
\end{definition}
\begin{example}\label{tpsd_eig}
\begin{enumerate}
\item Let $P = \X^2 \ominus \X \oplus \unit^{\circ}$. Then there is an infinite number of $\smax$-roots of $P$, since any $r$ with $|r|\leq \unit$ is a $\smax$-root of $P$.
However to be a $\smax^\vee$ root of $P$ (or corner root in \cite[Section 6]{adi}) one need that
$x^2\ominus x = x^2 \ominus x \oplus \unit^{\circ}\balance \zero$
 and the only solution is $\unit$.

\item Let $P=\X^3\oplus \X^2\oplus 2^\circ \X\oplus 2^\circ$. Then, again 
any $r$ with $|r|\leq \unit$ is a $\smax$-root of $P$.
However, $P$ has no $\smax^{\vee}$-root.
\end{enumerate}
\end{example}
\begin{definition}(Factorable polynomial fuction)
We say that the polynomial function $\widehat{P}$ can be factored (into linear factors) if there exist $r_i \in \smax^{\vee}$, for $i=1, \ldots, n$, such that 
\[ 
\widehat{P}(x)=  P_n (x \ominus r_1)  \cdots  (x \ominus r_n)\enspace .
\] 

\end{definition}

\begin{theorem}[Sufficient condition for factorization, see \protect{\cite[Th.\  4.4]{tavakolipour2021}}]\label{suf_cond}
Let ${P} \in \smax^\vee[\X]$. 
A sufficient condition for $\widehat{P}$  to be factored is that the formal polynomial $|{P}|$ is factored (see \Cref{roots_poly}).
In that case, we have  $\widehat{P}(x)= P_n (x \ominus r_1)  \cdots  (x \ominus r_n)$, with $n=\deg(P)$, $r_i\in\smax^\vee$, $i\in [n]$, such that $r_i P_{n-i+1}= \ominus P_{n-i}$ for all $i\leq n-\uval(P)$ and $r_i= \zero$ otherwise.
Moreover, $|r_1|\geq \cdots\geq |r_n|$ are the $\tmax$-roots  of  $|{P}|$,
counted with multiplicities.
\end{theorem}


\begin{corollary}[Sufficient condition for unique factorization, see \protect{\cite[Cor.\  4.6]{tavakolipour2021}}]\label{coro-uniquefact}
Let ${P} \in \smax^\vee[\X]$.
Assume that $|{P}|$ is factored (see \Cref{roots_poly}),
and let the $r_i$ be as in \Cref{suf_cond}.
If all the $r_i$ with same modulus are equal, 
or equivalently if for each $\tmax$-root $c\neq \zeror$ of $|{P}|$,
$c$ and $\ominus c$ are not both $\smax^\vee$-roots of $P$,
then the factorization of $\widehat{P}$ is unique (up to reordering).
\end{corollary}
The following definition of multiplicities of roots of polynomials was 
introduced in  \cite{baker2018descartes} in the framework of hyperfields, and
adapted in \cite[\S 5]{tavakolipour2021} to the more general framework of semiring systems. We write it below over $\smax$. Note that it only applies to polynomials with coefficients in $\smax^\vee$.

\begin{definition}[Multiplicity of $\smax^\vee$-roots, compare with \cite{baker2018descartes} and  \protect{\cite[\S 5]{tavakolipour2021}}] \label{def-mult-BL}
For a formal polynomial $P\in \smax^\vee[\X]$, 
and a scalar $r\in \smax^\vee$, we 
define the \new{multiplicity} 
of $r$ as a $\smax^{\vee}$-root of $P$, and denote it by $\mathrm{mult}_r(P)$, as follows.
If $r$ is not a root of $P$, set $\mathrm{mult}_r(P)=0$. 
If $r$ is a root of $P$, then 
\begin{equation}\label{mult}\mathrm{mult}_r(P)=1+\max\{\mathrm{mult}_r(Q)\mid Q\in \smax^\vee[\X],\; P \balance (\X \ominus r) Q\}\enspace .\end{equation}
\end{definition}

Characterization of multiplicities of polynomials over $\smax$
are given in \cite{tavakolipour2021} and in the work of Gunn~\cite{gunn,gunn2}.
In the special case of \Cref{coro-uniquefact}, the computations can be reduced as follows.
\begin{theorem}[Multiplicities and unique factorization, see \protect{\cite[Th.\ 6.7]{tavakolipour2021}}]\label{coro2-uniquefact}
Let ${P} \in \smax^\vee[\X]$ satisfy the conditions of \Cref{coro-uniquefact}.
Then the multiplicity of a $\smax^\vee$-root $r$ of $P$ coincides with the 
number of occurences of $r$ in the unique factorization of $\widehat{P}$.
It also coincides with the multiplicity of the $\tmax$-root $|r|$ 
of $|{P}|$.
\end{theorem}


\subsection{Eigenvalues and eigenvectors over $\rmax$, $\tmax$  and $\smax$}
\subsubsection{$\tmax$-eigenvalues}

When $\vgroup=\R$, the following definitions coincide with the ones used in~\cite{izhakianmatrix3,akian2016non}, for instance.
Let $A=(a_{ij}) \in  (\tmax)^{n \times n}$. Then, the $\tmax$-formal \new{characteristic polynomial} of $A$ is:
\[
P_A:=\per ( \X I\oplus A  )=\bigtsum_{k=0,\ldots,n}(P_A)_k \X^{k} \in \tmax[\X]
\enspace ,
\]
in which the expression of $\per (\X I \oplus A)$ is developped formally.
Equivalently, the coefficients of $P_A$ are given by
$(P_A)_k =\tsum_{I\subset [n],\; \card (I)=n-k} \per(A[I,I])$, 
where $A[I,I]$ is the submatrix of $A$ with rows and columns in $I$.
The polynomial function $\widehat{P_A}$
associated to $P_A$ is called the $\tmax$-\new{characteristic polynomial} function of $A$.

 \begin{definition}[$\tmax$-algebraic eigenvalue] \label{algebraic}Let $A \in (\tmax)^{ n \times n}$. The $\tmax$-\new{algebraic eigenvalues} of $A$, denoted by $\mu_{1}(A)\geq \cdots\geq \mu_{n}(A)$, are the $\tmax$-roots of its $\tmax$-characteristic polynomial.
\end{definition}
The term algebraic is used here since a $\tmax$-algebraic eigenvalue $\mu$ may not satisfy the eigenvalue-eigenvector equation $A u = \mu  u$ for some $u \in (\tmax)^{n},\; u \neq \zero$.
Nevertheless, the maximal such an eigenvalue $\mu$ is equal to the maximal algebraic eigenvalue $\mu_{1}(A)$ and is also equal to
the maximal cycle mean of $A$.
 The $\tmax$-characteristic polynomial function and therefore the $\tmax$-algebraic eigenvalues of $A \in (\tmax)^{n \times n}$ can be computed in $O(n^4)$ \cite{burkard2003finding} which can be reduced to $O(n^3)$, using parametric optimal assignment techniques \cite{gassner2010fast}.
However, no polynomial algorithm is known to compute all the coefficients of the $\tmax$-formal characteristic polynomial $P_A$  (see e.g.~\cite{butkovivc2007job}).
The computational complexity of computing the $\tmax$-eigenvalues can 
be reduced to polynomial time when considering special classes of matrices, 
such as symmetric matrices over $\{0,-\infty\}$, pyramidal matrices, Monge and Hankel matrices, tridiagonal Toeplitz and pentadiagonal Toeplitz matrices (see \cite{butkovivc2007job}, \cite{tavakolipour2020asymptotics}, \cite{tavakolipour2018tropical}).

As said before, for a general algebraic eigenvalue $\mu$, 
there may not exists a 
vector $u \in (\tmax)^{n},\; u \neq \zero$ such that $A u = \mu  u$.
Generalizations of the notion of eigenvectors have been considered in
\cite{izhakianmatrix3}, by replacing 
the equalities in $A u = \mu  u$ by the
conditions ``the maximum is attained at least twice'',
and  are handled by embedding $\tmax$ into the
supertropical semiring of Izhakian \cite{IR}.
More special generalizations have been considered in 
\cite{Nishida2020,Nishida2021,nishida2021independence},
where a constructive change of side of terms in each 
equation of  $A u = \mu  u$ is given, and depend on the eigenvalue $\mu$.
In the next section, we
shall consider another extension which uses signs and thus
the embedding of $\tmax$ into $\smax$.

\subsubsection{$\smax$-eigenvalues and $\smax$-eigenvectors}\label{subsec:eigvec}

\begin{definition}[$\smax$-formal characteristic polynomial]\label{charpoly_s}
The $\smax$-\new{formal characteristic polynomial} of $A \in (\smax)^{n \times n}$ is 
$\ps:= \det( \X I\ominus A ) \in \smax[\X]$,
and its $\smax$-\new{characteristic polynomial function} is 
$\widehat{P}_A(x) := \det(x I\ominus A)$.
\end{definition} 
We can also 
 write the coefficients of $\ps$ in terms of compound matrices of $A$. 
\begin{definition}($k$-th compound)\label{def-compound}
For $k \in [n]$, 
the $k$-th \new{compound} of a matrix $A \in (\smax)^{n \times n}$ is the matrix $\ext^k A \in (\mathbb{S}_{\max})^{{n\choose k} \times {n \choose k}}$ whose rows and columns are indexed by the subsets $K$ and $K'$ of $[n]$ of cardinality $k$ ($\mathrm{card}(K)=\mathrm{card}(K')=k$), and whose entries are 
$\bigg(\ext^k A\bigg)_{K,K'}= \det(A[K,K'])$
where $A[K,K']$ is the $k \times k$ submatrix obtained by selecting from $A$ the rows $i \in K$ and columns $j \in K'$.
We also set $\ext^0 A $ to be the $1\times 1$ identity matrix.
\end{definition}
\begin{definition}($k$-th trace)\label{def-trk}
The $k$-th trace of  $A \in (\smax)^{n \times n}$ is defined as
\[\tr_{k} A =\tr\bigg(\ext^k A\bigg) = \bigtsum_{\substack{K \subset [n]\\\mathrm{card}(K)=k}} \det(A[K,K])\enspace ,\]
for all $k \in [n]$, where $\ext^k A$ is the $k$-th compound of $A$,
see \Cref{def-compound}. 
\end{definition}

\begin{lemma}\label{comp_charpoly}
For $A \in (\smax)^{n \times n}$ we have
\[P_A = \bigtsum_{k=0,\ldots, n} \bigg((\ominus \unit)^{n-k} \tr_{n-k}A\bigg)
\X^{k}\enspace .\]
\end{lemma}
\Cref{charpoly} is an example of computation of the $\smax$-characteristic polynomial by using \Cref{comp_charpoly}.
\begin{definition}[$\smax$ and $\smax^\vee$-algebraic eigenvalues and their multiplicity]\label{s_eig}
Let $A \in (\smax)^{n \times n}$. 
Then,
 the $\smax$-roots (resp.\ $\smax^\vee$-roots) of $P_A$ (see \Cref{def-smaxroots}) 
are called the \new{$\smax$ (resp.\ $\smax^\vee$)-algebraic eigenvalues} of $A$.
If the characteristic polynomial $P_A$ has coefficients in $\smax^\vee$, then 
the multiplicity of $\gamma$ as a $\smax^\vee$-root of $P_A$ is called the
\new{multiplicity} of $\gamma$ as a $\smax$ (or $\smax^\vee$)-algebraic eigenvalue of $A$.
\end{definition}
Here, we defined two different notions of eigenvalues of a matrix over $\smax$.
In \cite[Section 6]{adi}, ``eigenvalues over $\smax$'' were defined as the corner roots of the characteristic polynomial, which correspond to $\smax^\vee$-algebraic eigenvalues in our definition. 


\begin{definition}[$\smax$-geometric eigenvalues and eigenvectors]\label{eig_vec}
Let $A \in (\smax)^{n \times n}$. 
Let $ v \in (\smax^\vee)^{n}\setminus\{\zero\}$ and $\gamma\in \smax^\vee$.
We say that $v$ is a \new{$\smax$-eigenvector} of $A$ associated with the \new{$\smax$-geometric eigenvalue} $\gamma$ if 
\begin{equation}\label{smaxeigenvector}
A  v \balance  \gamma  v\enspace.\end{equation}
\end{definition}
Since the last equation is equivalent to 
$(A \ominus \gamma I)  v \balance  \zero$, the following property follows from
the property of homogeneous systems in $\smax$ 
recalled in \Cref{homo}.
\begin{theorem}\label{existence}
Let $A\in (\smax)^{n \times n}$ and  $\gamma\in \smax^\vee$.
Then, $\gamma$ is a $\smax$-algebraic eigenvalue
if and only if there exists a $\smax$-eigenvector $v\in (\smax^{\vee})^n\setminus\{\zero\}$ associated to $\gamma$:
$A v\balance \gamma v\enspace.$ \hfill \qed
\end{theorem}
This shows that $\gamma$ is a $\smax$-geometric eigenvalue
if and only if it is a $\smax$-algebraic eigenvalue, as in usual algebra.
Then $\gamma$ is called a \new{$\smax$-eigenvalue}.
Note however that, even when $P_A$ has coefficients in $\smax^\vee$, the
multiplicity of $\gamma$ as a $\smax^\vee$-geometric eigenvalue of $A$ 
is difficult to define since there are several notions of independence and thus
of dimension over $\smax$ (see for instance~\cite{akian2009linear}).

We can weaken or strengthen the notion of $\smax$-eigenvector as follows.
\begin{definition}\label{smaxeigenvector-ws}
Let $A \in (\smax)^{n \times n}$ and let $\gamma$ be a $\smax$-eigenvalue.
\begin{description}
\item[Weak eigenvector] If $v\in (\smax)^{n}$ has at least one coordinate in $\smax^\vee\setminus\{\zero\}$ and satisfies \eqref{smaxeigenvector} then we say that
$v$ is a \new{weak $\smax$-eigenvector}.
\item[Strong eigenvector] If $v\in (\smax^\vee)^{n}\setminus\{\zero\}$ 
satisfies $A  v = \gamma  v$, then we say that $v$ is a
\new{strong $\smax$-eigenvector} and that $\gamma$ is a \new{strong $\smax$-geometric eigenvalue}.
\end{description}
\end{definition}
Using the above definitions, we have that 
a strong $\smax$-eigenvector is necessarily a $\smax$-eigenvector, and
a $\smax$-eigenvector is necessarily a weak $\smax$-eigenvector.


\subsubsection{Some special $\smax$-eigenvectors}\label{spec-eig-vector}

One effective approach to compute a $\smax$-eigenvector associated to the $\smax$-eigenvalue $\gamma$ is to use the columns of the adjugate of the matrix $A \ominus \gamma I$. The following states this approach.

\begin{proposition}\label{lem-Bk}
Suppose that $A \in (\smax)^{n \times n}$, let $\gamma$ be a $\smax$-eigenvalue of $A$ and 
denote 
\[B=\gamma  I \ominus A \enspace .\]
Then
\begin{equation}\label{adj_vec}
A \,  B^{\mathrm{adj}}  \balance \gamma B^{\mathrm{adj}} 
 \enspace. \end{equation}
\end{proposition}
\begin{proof}
Since $\gamma$ is a $\smax$-eigenvalue of $A$, using \Cref{s_eig} we have 
$\det(B) \balance \zero$,  and by \Cref{adj}, we have 
\[B \, B^{\mathrm{adj}} \succeq^{\circ} \det(B) I \succeq^{\circ} \zero\enspace.\]
So
\[A \,  B^{\mathrm{adj}}   
\ominus \gamma  B^{\mathrm{adj}}   = B  B^{\mathrm{adj}}  
\balance \zero\enspace.\]
Then by \Cref{equality_balance}-\eqref{pro1}, we obtain \eqref{adj_vec}.
\end{proof}
Property \eqref{adj_vec} implies that all the columns of $B^{\mathrm{adj}}$ with at least one entry in $ \smax^\vee\setminus\{\zero\}$  are weak $\smax$-eigenvectors associated with the $\smax$-eigenvalue $\gamma$. In usual algebra, a necessary and sufficient condition to obtain an eigenvector in this way is that the  (geometric) eigenvalue be simple, or equivalently that the matrix $B$ has rank $n-1$. In $\smax$ a similar condition, namely that there exists at least one $n-1\times n-1$ minor of $B$ in $\smax^\vee\setminus\{\zero\}$, or equivalently
that $B^{\mathrm{adj}}$ has at least one entry in $\smax^\vee\setminus\{\zero\}$ is sufficient to obtain one weak $\smax$-eigenvector.
However, it may not be sufficient to obtain one $\smax$-eigenvector in this way.
Below we give a stronger condition which is sufficient.

 Let $C \in \smax^{n\times n}$. In the following by $C_{i,:}$ and $C_{:,j}$ we mean the $i$-th row of $C$ and the $j$-th column of $C$, respectively.
Moreover,  $C_{i,\hat{j}}$ (resp.\ $C_{\hat{i},j}$) stands for the submatrix
of  $C_{i,:}$ (resp.\ $C_{:,j}$) 
obtained by eliminating the $j$th column (resp.\ the $i$-th row).
Recall that $C[\hat{i},\hat{j}]$ is the submatrix of $C$ obtained by 
 eliminating the $i$-th row and the $j$th column.
 \begin{theorem}[A sufficient condition for geometric simple $\smax$-eigenvalue]\label{cond_unique}
Consider $A$, $\gamma$ and $B$ as in \Cref{lem-Bk},
and let $v$ be a $\smax$-eigenvector associated to the $\smax$-eigenvalue $\gamma$.
\begin{enumerate}
\item Assume that there exists an entry of $B^\adj$ which is invertible, that is
$B^\adj_{i,j}\in  \smax^{\vee}\setminus\{\zero\}$ for some $i,j\in [n]$. 
 Then, there exists $\lambda\in \smax^\vee\setminus\{\zero\}$ 
such that  $v\balance \lambda B^\adj_{:,j}$.
\item Assume there exists a column $j$ of $B^\adj$ that is non-zero and has only 
$\smax^\vee$ entries: 
$B^\adj_{:,j}\in  (\smax^{\vee})^{n} \setminus\{\zero\}$.
 Then $B^\adj_{:,j}$ is a $\smax$-eigenvector associated to the $\smax$-eigenvalue $\gamma$, and there exists $\lambda\in \smax^\vee\setminus\{\zero\}$  such that $v= \lambda  B^\adj_{:,j}$.
\end{enumerate}
\end{theorem}
\begin{proof}
First, $v$ is a $\smax$-eigenvector associated to the $\smax$-eigenvalue $\gamma$ if and only if $v$ satisfies:
\begin{equation}\label{equofeigenvector}
v\in  (\smax^{\vee})^{n} \setminus\{\zero\}\quad\text{and}
\quad B v\nabla \zero\enspace .\end{equation}
Moreover if $j\in [n]$ is such that
$B^\adj_{:,j}\in  (\smax^{\vee})^{n} \setminus\{\zero\}$, 
then, by \Cref{lem-Bk}, we know that $B^\adj_{:,j}$ is a $\smax$-eigenvector associated to the $\smax$-eigenvalue $\gamma$ and thus a solution of \eqref{equofeigenvector}.

Proof of i): Let $i,j\in [n]$ be such that
$B^\adj_{i,j}\in \smax^{\vee}\setminus\{\zero\}$.
Denote $F:=B[\hat{j},\hat{i}]$, $b:=B_{\hat{j},i}$.
Denote also $P$ and $Q$ the permutation matrices associated to the
cycles $(1,\ldots, i)$ and $(1,\ldots, j)$ respectively.
Then applying these permutations on $B$, we obtain:
\begin{equation}\label{bprime}
B':= QBP^{-1}=\begin{pmatrix} 
* & *\\
b & F\end{pmatrix}
\enspace.\end{equation}
Applying the corresponding permutation on $v$, we
obtain $v':= P v= \begin{pmatrix}
v_{i}\\ \tilde{v}
\end{pmatrix}$ where
$\tilde{v}$ is obtained by eliminating the $i$-th entry of $v$.
Then, we have:
\begin{equation}\label{main_equ1}
 B v\nabla \zero\Leftrightarrow B'v'\nabla \zero \Rightarrow F \tilde{v} \nabla \ominus v_i b
\enspace .\end{equation}

We claim that 
\begin{equation}\label{formula-adj2}
\begin{pmatrix}\det(F)\\ \ominus F^{\adj}  b 
\end{pmatrix}
= (\ominus \unit )^{i+j} P B^\adj_{:,j} \enspace .\end{equation}

Let us first assume that \eqref{formula-adj2} holds and 
show that any $\smax$-eigenvector $v$ associated to the $\smax$-eigenvalue $\gamma$, or equivalently any solution of \eqref{equofeigenvector} satisfies 
 $v\balance \lambda  B^\adj_{:,j}$ for some $\lambda\in \smax^\vee\setminus\{\zero\}$.


Indeed, by \eqref{main_equ1}, any solution $v$ of  \eqref{equofeigenvector}
satisfies necessarily the equation
 $F \tilde{v} \nabla \ominus v_i b$. Then, applying the
 first part of  \Cref{cramer}  (Cramer's theorem),
we deduce that $\det(F)  \tilde{v} \balance  F^{\adj}  (\ominus v_i b)
= \ominus v_i  F^{\adj} b$.
Since $B^\adj_{i,j}\in \smax^{\vee}\setminus\{\zero\}$, 
it is invertible, and it follows for instance from 
\eqref{formula-adj2} that $\det(F)= (\ominus \unit )^{i+j} B^\adj_{i,j}$
so is invertible. 
 So, $\tilde{v} \balance \det(F)^{ -1} (\ominus v_i  F^{\adj} b)$.
Using \eqref{formula-adj2}, we obtain that
$Pv \balance \det(F)^{ -1} v_i \begin{pmatrix}\det(F)\\ \ominus F^{\adj}  b 
\end{pmatrix}=  \det(F)^{ -1} v_i (\ominus \unit )^{i+j} P B^\adj_{:,j} $.
Therefore $v\balance \det(F)^{ -1} v_i (\ominus \unit )^{i+j} B^\adj_{:,j} $.
In particular, if $v_i=\zero$, then $v\balance \zero$ and so $v$ is 
not in $(\smax^{\vee})^{n} \setminus\{\zero\}$, a contradiction
with  \eqref{equofeigenvector}.
Therefore $v_i\in \smax^\vee\setminus\{\zero\}$, and we get that any
solution  $v$ of \eqref{equofeigenvector} 
satisfies $v\balance \lambda B^\adj_{:,j}$ for $\lambda=\det(F)^{ -1} v_i (\ominus \unit )^{i+j} \in \smax^\vee\setminus\{\zero\}$.

Let us now show our claim, that is \eqref{formula-adj2}.
First, we have that $(B')^\adj= (P^{-1})^{\adj}B^\adj Q^\adj
= \det(P)^{-1} P B^\adj \det(Q) Q^{-1}$ since
$P$ and $Q$ are invertible matrices (see for instance \cite[Cor.\ 2.35]{adi}).
Therefore, we have $(B')^\adj_{:,1}= (\ominus \unit)^{i+j} (P B^\adj)_{:, j}$,
which is the right hand side of  \eqref{formula-adj2}.
The coordinates of  $w=(B')^\adj_{:,1}$ are
$w_k=(B')^\adj_{k,1}=(\ominus \unit)^{k+1} \det (B'[\hat{1},\hat{k}])$,
$k\in [n]$.
Using \eqref{bprime}, we have clearly $w_1=\det(F)$.

For $k\in [n-1]$, let us denote by $F_k$ the matrix obtained from $F$ 
after replacing its $k$-th column with $b$.
Then, by \Cref{ith_cramer}, we have that $(F^\adj b)_k= \det(F_k)$.
Let $B'[\hat{1},:]$ be the matrix obtained from $B'$ after 
eliminating the first row, we have $B'[\hat{1},:]= \begin{pmatrix}b & F\end{pmatrix}$.
Since $b$ is the first column of this matrix, we have that 
$F_k$ can also be obtained from the matrix $B'[\hat{1},:]$
after  eliminating the $k+1$ column, then doing $k-1$ swaping from the first column to the $k$-th column.  So, $\det(F_k)=(\ominus \unit )^{k-1}\det(B'_{\hat{1},\widehat{k+1}})$ and therefore, we have 
\[ \ominus (F^\adj b)_k=\ominus \det(F_k)= (\ominus \unit)^{k}\det(B'_{\hat{1},\widehat{k+1}})= (B')^\adj_{k+1,1}\enspace .\]

Proof of ii): If now $B^\adj_{:,j}\in (\smax^\vee)^n\setminus\{\zero\}$,
with $(B^\adj)_{i,j}\neq \zero$, then Point i) shows that
$v\balance \lambda B^\adj_{:,j}$ for $\lambda=\det(F)^{ -1} v_i (\ominus \unit )^{i+j} \in \smax^\vee\setminus\{\zero\}$.
Since both sides of the balance equations are in $\smax^\vee$, the second part of \Cref{equality_balance} implies the equality, and so we get that
$v= \lambda B^\adj_{:,j}$, which finishes the proof of \eqref{formula-adj2}.
Note that this second part of the theorem can also be shown using
the second part of  \Cref{cramer}  (Cramer's theorem).
\end{proof}

\begin{theorem}\label{cond_existence}
Let $A$, $\gamma$ and $B$ as in \Cref{lem-Bk}. Assume that there exists an entry of $B^\adj$ which is non-zero, that is
$B^\adj_{i,j}\neq \zero$ for some $i,j\in [n]$.  Then there exists a $\smax$-eigenvector $v$ associated to the $\smax$-eigenvalue $\gamma$ such that $|v|=|B^{\adj}_{:,j}|$ and $v_i=B^{\adj}_{i,j}$ for all $i\in [n]$ satisfying $B^{\adj}_{i,j}\in\smax^\vee$.
\end{theorem}
\begin{proof}
Using the same arguments and notations
 as in the proof of first point of \Cref{cond_unique},
we have that $v$ is a $\smax$-eigenvector $v$ associated to the $\smax$-eigenvalue $\gamma$ if and only the vector $\tilde{v}$ satisfies \eqref{main_equ1}.
Moreover, $\det(F)= (\ominus \unit )^{i+j} B^\adj_{i,j}$, so that 
$\det(F)\neq \zero$. 
Applying \Cref{existence_signed}, we get that for any $v_i\in\smax^\vee$,
there exists $\tilde{v}$ satisfying \eqref{main_equ1}
and
$|\tilde{v}|=|\det(F)|^{-1} |F^\adj ( \ominus v_i b)|$.
Using again the same arguments as in the proof of  first point of \Cref{cond_unique}, we deduce that
$|Pv|=|v_i| |\det(F)|^{-1} |P B^\adj_{:,j}|$.
Since $P$ is a permutation matrix, choosing $v_i= |\det(F)|$,
we obtain $|v|= |B^\adj_{:,j}|$.

Now by the first point of \Cref{cond_unique}, we know that there exists
$\lambda\in\smax^\vee\setminus\{\zero\}$ such that
$v\balance \lambda B^\adj_{:,j}$. If there exists $i\in [n]$
such that $B^{\adj}_{i,j}\in\smax^\vee\setminus\{\zero\}$, then 
by the second point of \Cref{equality_balance}, we have 
$v_i=\lambda B^{\adj}_{i,j}$ and since $|v_i|=|B^{\adj}_{i,j}|$, 
we deduce that $\lambda=\unit$ or $\ominus\unit$.
Replacing $v$ by $\lambda^{-1} v$, we get that $v$ is a
$\smax$-eigenvector $v$ associated to the $\smax$-eigenvalue $\gamma$
and is such that $v\balance B^\adj_{:,j}$ and $|v|= |B^\adj_{:,j}|$.
Using again  the second point of \Cref{equality_balance}, we deduce that
 $v_i=B^{\adj}_{i,j}$ for all $i\in [n]$ such that $B^{\adj}_{i,j}\in\smax^\vee$.
\end{proof}
\section{Tropical positive (semi-)definite matrices and their eigenvalues}\label{sec:3}
Tropical positive semi-definite matrices were introduced in \cite{yu2015tropicalizing} and generalized in \cite{tropicalization}. Here we consider also tropical  positive definite matrices. 

\subsection{Tropical positive (semi-)definite matrices}
\begin{definition}[$\pd$ and $\psd$ matrices, compared with \cite{tropicalization}]\label{def:psd}
Let $A=(a_{ij} ) \in (\smax^\vee)^{n \times n}$ be a symmetric matrix. It is said to be \new{tropical positive definite} ($\pd$) if 
  \begin{equation}\label{def_pd}\zero \lsign x^{T} A  x,\;
\text{that is}\; x^{T} A  x \in \smax^{\oplus}\setminus\{\zero\},\;
\text{for all}\; x \in (\smax^{\vee})^{n}\setminus\{\zero\}\enspace.\end{equation}
  If the strict inequality required in \Cref{def_pd} is weekened to $\zero \leqsign x^{T} A  x$, then $A$ is said to be \new{tropical positive semi-definite} ($\psd$). 
\end{definition}
Throughout the paper, the set of $n\times n$ $\pd$ and $\psd$ matrices over $\smax^{\vee}$, are denoted by $\pd_n(\smax^{\vee})$ and $\psd_n(\smax^{\vee})$, respectively. Therefore we have $\pd_n(\smax^{\vee}) \subseteq \psd_n(\smax^{\vee})$.

We recall in \Cref{def_psd1} below 
the characterization of tropical positive definite matrices 
shown in \cite{tropicalization}. 


\begin{theorem}[\cite{tropicalization}]\label{def_psd1}
The set $\psd_{n}(\smax^\vee)$ is equal to the set 
\[ 
 \{A=(a_{ij}) \in (\smax^{\vee})^{n \times n} : \zero \leqsign a_{ii}\; \forall i \in [n],\; a_{ij}=a_{ji} \;\text{and}\; a_{ij}^{ 2} \leqsign a_{ii} a_{jj}\; \forall i,j \in [n], i \neq j\}\enspace . \]
\end{theorem}
Using \Cref{def_psd1}, one can obtain the following similar result for $\pd$ matrices. We give a detailed proof in Appendix.
\begin{theorem}\label{def_pd1}
The set $\pd_{n}(\smax^\vee)$ is equal to the set 
\[ 
 \{A=(a_{ij}) \in (\smax^{\vee})^{n \times n} : \zero \lsign a_{ii}\; \forall i \in [n],\; a_{ij}=a_{ji} \;\text{and}\; a_{ij}^{ 2} \lsign a_{ii} a_{jj}\; \forall i,j \in [n], i \neq j\}\enspace . \]
\end{theorem}






Note that, in the above characterizations of $\psd$ and $\pd$ matrices, 
the inequalities involve diagonal entries or the square of non-diagonal entries,
which are all elements of $\smax^{\oplus}$.

\subsection{The $\smax$-characteristic polynomial of $\psd$ and $\pd$ matrices}
The following result will help us to compute the characteristic polynomial.
\begin{theorem}\label{trace}
Let $A \in \psd_n(\smax^{\vee})$ with the diagonal elements $d_n \leqsign \cdots \leqsign d_1$. Then, we have 
$\tr_k A=   \bigtprod_{i\in [k]}d_i \;\text{or} \;\tr_kA  =( \bigtprod_{i\in [k]}d_i)^{\circ}$, so $\tr_k A\geq 0$,
 and for $A \in \pd_n(\smax^{\vee})$ we have  
 $\tr_kA=   \bigtprod_{i\in [k]}d_i> 0$.
\end{theorem}
The proof will follows from the following lemmas.
\begin{lemma}\label{diag_cycle}
Let $A=(a_{ij}) \in \psd_n(\smax^{\vee})$. 
Let $\cycle$ be a cycle $(j_{1},j_{2},\ldots ,j_{k})$ of length $k>1$ in $[n]$
and let us denote by $[\cycle]=\{j_{1},j_{2},\ldots ,j_{k}\}$ the set of its elements. Then 
\begin{enumerate}
\item $|w(\cycle)| \leqsign \bigtprod_{i\in [\cycle]}a_{ii}.$
\item Moreover, if $A\in \pd_n(\smax^{\vee})$ we have 
$|w(\cycle)| \lsign \bigtprod_{i\in [\cycle]}a_{ii}$.
\end{enumerate}
\end{lemma}
\begin{proof}
{\bf Proof of Part 1}: Let $\cycle$ be the cycle $(j_{1},j_{2},\ldots ,j_{k})$. Since $A \in \psd_n(\smax^{\vee})$ by  \Cref{def_psd1} we have
\[\begin{array}{ccc}
a_{j_1j_2}^{ 2}&\leqsign& a_{j_1j_1} a_{j_2j_2}\enspace,\\
a_{j_2j_3}^{ 2}&\leqsign &a_{j_2j_2}  a_{j_3j_3}\enspace,\\
&\vdots&\\
a_{j_kj_1}^{ 2}&\leqsign& a_{j_kj_k} a_{j_1j_1}\enspace.\end{array}
\]
So, by the first part of \Cref{product_order} we have
$ a_{j_1j_2}^{ 2} a_{j_2j_3}^{ 2}  \cdots  a_{j_kj_1}^{ 2} \leqsign a_{j_1j_1}^{ 2} a_{j_2j_2}^{ 2}  \cdots  a_{j_kj_k}^{ 2}\enspace$.
Finally, using \Cref{modulus_order},
\begin{eqnarray}
|a_{j_1j_2} a_{j_2j_3}  \cdots  a_{j_kj_1}|&\leqsign&
 |a_{j_1j_1} a_{j_2j_2} \cdots  a_{j_kj_k}| \nonumber\\\label{mar2}
&=& a_{j_1j_1} a_{j_2j_2} \cdots  a_{j_kj_k} \enspace,\nonumber
\end{eqnarray}
where the last equality is due to the positiveness of diagonal elements of $A$.\\
{\bf Proof of Part 2}: The proof of the Part 2 is obtained similarly by using the definition of $\pd$ matrices instead of $\psd$ matrices and the second part of \Cref{product_order}.
\end{proof}

\begin{lemma}\label{diag_cycle2}
Let $A=(a_{ij}) \in \psd_n(\smax^{\vee})$.
Let $\permutation$ be any permutation of $[n]$.
Then 
\begin{enumerate}
\item $|w(\permutation)| \leqsign \bigtprod_{i\in [n]}a_{ii},$ with equality when
$\permutation$ is the identity permutation.
\item Moreover, if $A\in \pd_n(\smax^{\vee})$ and $\permutation$ is different from
the identity permutation, we have 
$|w(\permutation)| \lsign \bigtprod_{i\in [n]}a_{ii}.$
\end{enumerate}
\end{lemma}
\begin{proof}
Since every permutation of $[n]$ can be decomposed uniquely
into disjoint cycles which cover $[n]$,
Part 1 of \Cref{diag_cycle} is true for any permutation, when replacing
$[\cycle]$ by $[n]$. 
Moreover, if the permutation is different from identity, then applying 
Part 2 of \Cref{diag_cycle} to all the cycles of length $>1$, we get 
Part 2 of \Cref{diag_cycle2}.
\end{proof}

\begin{proof}[Proof of \Cref{trace}]
Let $k \in [n]$ and $A \in \psd_n(\smax^{\vee})$. 
For any subset $K$ of $[n]$ with cardinality $k$,
the submatrix $A[K,K]$ is a positive semi-definite matrix over $\smax^\vee$.
Applying Part 1 of \Cref{diag_cycle2} to this matrix, we obtain
that $|\det(A[K,K])|=\bigtprod_{i\in K}a_{ii}$.
Then, by \Cref{def-trk}, and using that $d_1\geq \cdots\geq d_n$,
we get that $|\tr_kA|= \bigtprod_{i\in [k]}d_i$.
 Since, $\bigtprod_{i\in [k]}d_i$ is one of the summands in the formula of
$\tr_kA$, we have $\tr_k A\succeq \bigtprod_{i\in [k]}d_i$.
Therefore we conclude two possible cases:
$\tr_kA=   \bigtprod_{i\in [k]}d_i \;\text{or} \;\tr_kA  =( \bigtprod_{i\in [k]}d_i)^{\circ}$.
Also, for $A \in \pd_n(\smax^{\vee})$, and any subset $K$ of $[n]$ 
with cardinality $k$, 
the submatrix $A[K,K]$ is a positive definite matrix over $\smax^\vee$.
Therefore, applying Part 2 of \Cref{diag_cycle2} to this matrix,
we obtain that there is no permutation $\permutation$ of $K$ such that $|w(\permutation)|=\bigtprod_{i\in K}a_{ii}$,
other than identity permutation. Hence,  
$\det(A[K,K])=\bigtprod_{i\in K}a_{ii}$.
Since all the terms $\det(A[K,K])$ are in $\smax^\oplus$, we get that
$\tr_kA$ is also in $\smax^\oplus$, and so  $\tr_kA=   \bigtprod_{i\in [k]}d_i$.
\end{proof}

\begin{corollary}\label{char_pd}
For $A=(a_{ij}) \in \pd_n(\smax^{\vee})$ with the diagonal elements $d_n \leqsign \cdots \leqsign d_1$ we have
\[ P_A = \bigtsum_{k=0}^{n} \bigg((\ominus \unit)^{n-k}  (\bigtprod_{i\in [n]-k}d_i)\bigg)\X^{k}\enspace .\]
\end{corollary}

\begin{example}\label{balanc_char}
Let $A= \begin{pmatrix}
\unit&\unit\\
\unit&\unit
\end{pmatrix} \in \psd_2(\mathbb{S_{\max}^{\oplus}})$. By \Cref{comp_charpoly}, the formal characteristic polynomial of $A$ is 
$P_A = \X^2 \ominus \X \oplus \unit^{\circ}$,\;
which shows that the formal characteristic polynomial associated to $\psd$ matrices may have balance elements. In \Cref{tpsd_eig} we considered the $\smax$-roots and $\smax^{\vee}$-roots of $P_A$ which are the same as $\smax$-eigenvalues and $\smax^{\vee}$-eigenvalues of $A$. 
\end{example}

\begin{remark}
In usual algebra, semi-definite matrices which are not definite have the eigenvalue 0, here this is replaced by the fact that the characteristic polynomial
has a balanced constant coefficient and that there is an infinite number of $\smax$-eigenvalues.
\end{remark}


 \begin{example}\label{charpoly}
 Let 
$A = \begin{pmatrix}
3 &2& 1\\
2&2&1\\
1&1&1
\end{pmatrix}$. We have $A \in \pd_{3}(\smax^{\vee})$ and
$\ext^1 A =\begin{pmatrix}
3 &2& 1\\
2&2&1\\
1&1&1
\end{pmatrix}
$,
\[\begin{array}{ccc}
\ext^2 A& =&\begin{pmatrix}
\det\begin{pmatrix}
3&2\\2&2
\end{pmatrix}
 &\det\begin{pmatrix}
3&1\\2&1
\end{pmatrix}
& 
\det\begin{pmatrix}
2&1\\2&1
\end{pmatrix}
\\[1em]
\det\begin{pmatrix}
3&2\\1&1
\end{pmatrix}
&
\det\begin{pmatrix}
3&1\\1&1
\end{pmatrix}
&
\det\begin{pmatrix}
2&1\\1&1
\end{pmatrix}
\\[1em]
\det\begin{pmatrix}
2&2\\1&1
\end{pmatrix}
&\det\begin{pmatrix}
3&1\\1&1
\end{pmatrix}
&\det\begin{pmatrix}
2&1\\1&1
\end{pmatrix}
\end{pmatrix}
=\begin{pmatrix}
5 &4& 3^{\circ}\\
4&4&3\\
3^\circ&4&3
\end{pmatrix}, 
\end{array}\]
and 
$\ext^3 A =\det\begin{pmatrix}
3 &2& 1\\
2&2&1\\
1&1&1
\end{pmatrix}=6$.
Therefore
$\tr_{0} A=\unit, \; \tr_{1} A= 3, \; \tr_{2} A= 5$ and $\tr_{3} A=6.$
So, we have
$P_A = \X^3 \ominus 3 \X^2 \oplus 5\X \ominus 6\enspace$\enspace.
\Cref{Fig:plot_poly}
illustrates the plot of $P_A$.
\begin{figure}[!h]
\small
 \centering
\begin{tikzpicture}[scale=0.7]
\draw[->] (-3.5,0) -- (3.5,0);
\draw[->] (0,-6.5) -- (0,6.5);
\draw[dotted](1,-1) -- (1,1);
\draw[dotted] (2,-2) -- (2,2);
\draw[dotted] (3,4) -- (3,-4);
\draw[thick] (1,-1) -- (-1,-1);
\draw[thick] (-1,-1) -- (-2,-2);
\draw[thick] (-2,-2) -- (-3,-4);
\draw[thick] (1,1) -- (2,2);
\draw[thick] (2,-2) -- (3,-4);
\draw[thick] (3,4) -- (3.5,6.5);
\draw[thick] (-3,-4) -- (-3.5,-6.5);

\fill (1,1) circle (3pt);
\fill (1,-1) circle (3pt);
\fill (3,4) circle (3pt);
\fill (3,-4) circle (3pt);
\fill (2,2) circle (3pt);
\fill (2,-2) circle (3pt);
\fill (-1,-1) circle (3pt);
\fill (-2,-2) circle (3pt);
\fill (-3,-4) circle (3pt);
\fill (0.25,-0.25) node {\tiny$\zero$};
\fill (-4,-0.4) node {\tiny$\smax^{\ominus}$};
\fill (4,-0.4) node {\tiny$\smax^{\oplus}$};
\fill (0.5,6) node {\tiny$\smax^{\oplus}$};
\fill (0.5,-6) node {\tiny$\smax^{\ominus}$};
\fill (-1,-0.4) node {\tiny$\ominus 1$};
\fill (-2,-0.4) node {\tiny$\ominus 2$};
\fill (-3,-0.4) node {\tiny$\ominus 3$};
\fill (1.1,-0.4) node {\tiny$1$};
\fill (2.1,-0.4) node {\tiny$2$};
\fill (3.1,-0.4) node {\tiny $3$};
\fill (0.25,-1) node {\tiny$\ominus 6$};
\fill (0.25,-2) node {\tiny$\ominus 7$};
\fill (0.25,-4) node {\tiny$\ominus 9$};
\fill (0.25,1) node {\tiny$6$};
\fill (0.25,2) node {\tiny$7$};
\fill (0.25,4) node {\tiny$9$};
\end{tikzpicture}\caption{ Plot of $P_A=\X^3 \ominus 3 \X^2 \oplus 5\X \ominus 6$ in \Cref{charpoly}. The solid black line illustrates $\widehat{P_A}$. The points of
discontinuity of $\widehat{P_A}$ are $1, 2$ and $3$ which are the roots of $P_A$\enspace.  }\label{Fig:plot_poly}
 \end{figure}
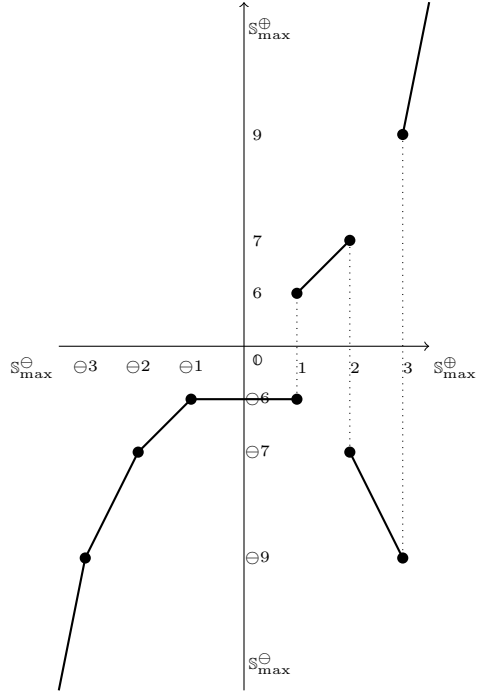

\end{example}




\subsection{$\tmax$-Eigenvalues and $\smax$-Eigenvalues of  $\psd$ and $\pd$ matrices}\label{sec:eig}
Let $A$ be a  $\psd$ matrix. In the following theorem, we compute the $\tmax$-eigenvalues of $|A|$.
\begin{theorem}\label{tropical_eigs}
Let $A=(a_{ij}) \in \psd_n(\smax^{\vee})$. Then the $\tmax$-eigenvalues of $|A|=(|a_{ij}|)\in (\tmax)^{n \times n}$ are the diagonal elements of $|A|$ counted with multiplicities.
\end{theorem}

\begin{proof}
Let $d_1\geq d_2\geq \cdots \geq d_n$ be the diagonal elements of $|A|$. W.l.o.g let $d_1 \neq \zero$. Therefore, for $i\in [n]$ we get that $\tr_iA\neq \zero$. Otherwise $d_1= \cdots=d_n=\zero$ and since $A \in \psd_n(\smax^{\vee})$ we have $A=\zero$ and the proof is straightforward.  Using 
\Cref{perdet} the characteristic polynomial of $|A|$ over $\tmax$ is
$P_{|A|} = \tsum_{k=0,\ldots,n}  (\tr_{n-k}A)\X^{k}$.
 By \Cref{trace} for $k=2, \ldots, n$ 
\[d_{k-1}= \tr_{k-1}A-\tr_{k-2}A \geq  \tr_{k}A - \tr_{k-1}A=d_k.\]
Finally, using \Cref{concavepoly} together with \Cref{roots_poly} we deduce the result.
\end{proof}


Let us consider \Cref{balanc_char} again. The $\smax$-charactestic polynomial of $A$ has the polynomial function $\widehat{P}_A(x) = x^2 \oplus \unit^{\circ} $ which is not a polynomial function in $\PF (\smax^{\vee})$.
So we are not interested in considering the $\smax$-eigenvalues of $\psd$ matrices. From here on we prove our results only for the case of $\pd$ matrices.

\begin{theorem}\label{sym_eigs}
  Let $A \in \pd_n(\smax^{\vee})$. The diagonal elements of $A$ are precisely
  the $\smax$-eigenvalues of $A$, counted with multiplicities.
\end{theorem}

\begin{proof}
Let $d_1\geq d_2\geq \cdots \geq d_n$ be the diagonal elements of $A$. Using \Cref{char_pd} we have 
\begin{equation}\label{factor_poly}P_A(\X)= \bigtsum_k ((\ominus \unit )^{ k}  d_1  \cdots  d_k) \X^{n-k}\end{equation}
and therefore by \Cref{concavepoly} and \Cref{roots_poly} we have
\[|P_A|(\X)= \bigtsum_k (d_1  \ldots  d_k )\X^{n-k}= (\X \oplus d_1)  \cdots (\X \oplus d_n). \]
Moreover, using \Cref{factor_poly} we have $P_{n-i+1}= (\ominus \unit)^{i-1}\tr_{i-1}A$ and 
$\ominus P_{n-i} = (\ominus \unit)^{i+1} \tr_iA$. Therefore $d_i  P_{n-i+1}= \ominus P_{n-i}$ and by \Cref{suf_cond}, $d_i,\; i\in [n]$ are the $\smax$-roots of $P_A$. Also since all the diagonal elements of $A$ ($d_i,\; i\in [n]$) are positive,  \Cref{coro-uniquefact} and \Cref{coro2-uniquefact}
give us that $P_A$ has a unique factorization and that the multiplicity
of a diagonal element as a $\smax$-eigenvalue of $A$ coincides
with the number of its occurences as a diagonal element.
\end{proof}

\section{Eigenvectors of tropical positive (semi-)definite matrices}\label{sec:3p}
\subsection{$\smax$-Eigenvectors of $\pd$ matrices using the adjoint matrix}

We now specialize some of the properties proved in \Cref{spec-eig-vector}

\begin{proposition}\label{balance-adj}
Let $A\in \pd_n(\smax^\vee)$, and set $\gamma_{i}=a_{ii}$ for $i\in [n]$.
Assume that $\gamma_{1}\succeq \gamma_{2} \succeq \cdots \succeq \gamma_{n}$,
and define $B_k=\gamma_k I\ominus A$ for some $k \in [n]$.
Then, all the diagonal entries of $(B_k)^{\mathrm{adj}}$ are non-zero and
they are all in $\smax^\circ$
except possibly the $k$-th diagonal entry,
which is also in $\smax^\circ$ if and only if $\gamma_k$ is not a simple $\smax$-eigenvalue.

\end{proposition}
\begin{proof}
Note that all the $\gamma_k$ are different from $\zero$.
Indeed, the modulus of $B_k$ is a positive semi-definite matrix 
with diagonal entries equal to $\gamma_{1},\ldots, \gamma_{k-1}, \gamma_k,\ldots,
\gamma_{k}$. So all $(n-1)\times (n-1)$ principal submatrices are also 
of same type and so have a determinant modulus equal to
the product of its diagonal entries moduli.
Since the determinant is also $\succeq$ to this product,
it is non-zero and 
we get that it is in $\smax^\circ$, if the product is in $\smax^\circ$.
This is the case for all the principal submatrices which contain the $k$-th 
diagonal element of $B_k$. This is also the case, when $\gamma_k$ is not 
a simple $\smax$-eigenvalue.
If 
$\gamma_k$ is a simple $\smax$-eigenvalue,
then one can show that the $k$th diagonal entry of
$(B_k)^{\mathrm{adj}}$ is 
equal to  $(\ominus \unit)^{k-1} \gamma_{1}\cdots \gamma_{k-1} \gamma_k^{n-k}$,
so is not in $\smax^\circ$.
\end{proof}
Note that  $\gamma_k$ is a simple $\smax$-eigenvalue if and only if $\gamma_{k-1}\succ \gamma_k \succ \gamma_{k+1}$, with the convention $\gamma_{n+1}=\zero$.
By \Cref{lem-Bk}, special weak $\smax$-eigenvectors associated to
the eigenvalue $\gamma_k$ are the columns of $(B_k)^{\mathrm{adj}} $
which are not in $(\smax^\circ)^n$.
When $\gamma_k$ is simple, the above result shows that among the columns of $(B_k)^{\mathrm{adj}} $, the $k$-th column is necessarily a weak $\smax$-eigenvector
associated to $\gamma_k$, and that the other columns 
cannot be $\smax$-eigenvectors.
Hence,  the $k$-th column is the only candidate to be a  $\smax$-eigenvector,
we shall denote it by $v^{(k)}$.

\begin{corollary}\label{coro-simple-eigen}
Let $A\in \pd_n(\smax^\vee)$, and $\gamma=\gamma_k$ and $B=B_k$ be as in \Cref{balance-adj}. Assume that $\gamma$ is a simple $\smax$-eigenvalue. Let 
\begin{equation}\label{vk}
v^{(k)}:= (B_k)_{:,k}^{\mathrm{adj}}.
\end{equation}
Then we have the following properties:
\begin{enumerate}
\item   $v^{(k)}$ is a  weak $\smax$-eigenvector
associated to $\gamma$, such that  $v^{(k)}_k\in\smax^\vee\setminus\{\zero\}$.
\item There exists a $\smax$-eigenvector $v$ associated to $\gamma$ such that
$|v|=|v^{(k)}|$ and $v_i=v^{(k)}_i$ for all $i\in [n]$
satisfying $v^{(k)}_i\in\smax^\vee$, in particular for $i=k$.
\item Any $\smax$-eigenvector $v$ associated to $\gamma$ satisfies $v\balance \lambda v^{(k)}$ for some $\lambda\in \smax^{\vee}\setminus\{\zero\}$.
\end{enumerate}
\end{corollary}
\begin{proof}
Since $\gamma$ is simple,  \Cref{balance-adj} shows that $(B_k)_{k,k}^{\mathrm{adj}}\in \smax^\vee\setminus\{\zero\}$.
Then, Point i) follows from \Cref{lem-Bk}.

Point ii)  follows from \Cref{cond_existence}, using that $(B_k)_{k,k}^{\mathrm{adj}}\neq \zero$, and the fact that $i=k$ is possible follows from $(B_k)_{k,k}^{\mathrm{adj}}\in \smax^\vee\setminus\{\zero\}$.

Point iii) follows from the first part of~\Cref{cond_unique} using that $(B_k)_{k,k}^{\mathrm{adj}}\in \smax^\vee\setminus\{\zero\}$.
\end{proof}
\begin{remark}
If $\gamma$ is not necessarily simple, then Point ii) in \Cref{coro-simple-eigen} still holds, except that $i=k$ may not satisfy the property.
Indeed, this follows from \Cref{cond_existence}, using that $(B_k)_{k,k}^{\mathrm{adj}}\neq \zero$, and the later is always true for a positive definite matrix $A$.
Moreover, the same holds by replacing $v^{(k)}$ by any column of 
$(B_k)^{\mathrm{adj}}$, since all diagonal entries of $(B_k)^{\mathrm{adj}}$ are non-zero, by \Cref{balance-adj}.
\end{remark}

\begin{corollary}\label{coro-unique-eigen}
Let $A\in \pd_n(\smax^\vee)$, and $\gamma=\gamma_k$ and $B=B_k$ be as in \Cref{balance-adj}. Assume there exists a column $j$ of $B^\adj$ which is in $(\smax^\vee)^n\setminus \{\zero\}$ (as in \Cref{cond_unique}).
Then, $j=k$, and any $\smax$-eigenvector is
a multiple of $B^\adj_{:,j}$ and $\gamma$ is a simple
(algebraic) $\smax$-eigenvalue of $A$.
\end{corollary}
\begin{proof}
Assume there exists a column $j$ of $B^\adj$ which is in $(\smax^\vee)^n\setminus \{\zero\}$.
\Cref{balance-adj} shows that any
column of $B^\adj$ different from the  $k$-th column has a non-zero
balanced coefficient, and so $j=k$. 
Also, if $\gamma$ is not simple, the same holds for the $j$-th column.
This shows that $\gamma$ is  a simple
(algebraic) $\smax$-eigenvalue of $A$.
Finally,  by the second part of~\Cref{cond_unique}, any $\smax$-eigenvector associated to the eigenvalue $\gamma$ is a multiple of $B^\adj_{:,k}$.
\end{proof}

In \Cref{ex_eig} and \Cref{ex_eig2}, we shall see that even though the entries of $A$ are in $\smax^{\vee}$, and that $A$ has $n$ distinct eigenvalues, there may exist eigenvalues $\gamma_k$ such $v^{(k)}$ (and thus any column of $(B_k)^{\mathrm{adj}}$) 
is not a $\smax$-eigenvector, and that this may hold for the maximal eigenvalue,
see  \Cref{ex_eig2}.


 \begin{example}\label{ex_eig1}Let 
$A = \begin{pmatrix}
3 &\ominus 2& 1\\
\ominus 2&2&1\\
1&1&1
\end{pmatrix}$.
It is immediate to see that $A \in \pd_n(\smax^\vee)$ with the $\smax$-eigenvalues $\gamma_1=a_{11}=3$, $\gamma_2=a_{22}=2$ and $\gamma_3 = a_{33}= 1$. 
We get
\[B_1 = \gamma_{1}  I \ominus A= \begin{pmatrix}
3^{\circ} &2& \ominus 1\\
 2& 3&\ominus1\\
\ominus 1&\ominus 1& 3
\end{pmatrix}
\Rightarrow (B_1)^{\mathrm{adj}}=
\begin{pmatrix}
\mathbf{6}&\ominus 5&4\\
\mathbf{\ominus 5}&6^{\circ}&4^{\circ}\\
\mathbf{4}&4^{\circ}&6^{\circ}
\end{pmatrix} \Rightarrow v^{(1)} = \begin{pmatrix} 6\\\ominus 5\\4\end{pmatrix}\]
For the $\smax$-eigenvector associated to $\gamma_2=a_{22}$ we have 
\[B_2=  \gamma_{2}  I \ominus A  =\begin{pmatrix}
\ominus 3 & 2& \ominus 1\\
 2&2^{\circ} &\ominus 1\\
\ominus 1&\ominus 1& 2
\end{pmatrix} \Rightarrow(B_2)^{\mathrm{adj}}=\begin{pmatrix}
4^{\circ} &\mathbf{\ominus 4}&3^{\circ}\\
\ominus 4&\mathbf{\ominus 5}&\ominus 4\\
3^{\circ}&\mathbf{\ominus 4}&5^{\circ}
\end{pmatrix}\Rightarrow v^{(2)} = \begin{pmatrix} \ominus  4\\\ominus 5\\\ominus 4\end{pmatrix} \]
Also, we have 
\[B_3=\gamma_{3}  I \ominus A  = \begin{pmatrix}
\ominus 3 & 2& \ominus 1\\
2& \ominus  2 & \ominus 1\\
\ominus 1& \ominus 1&1^{\circ}
\end{pmatrix}\Rightarrow (B_3)^{\mathrm{adj}}=\begin{pmatrix}
3^{\circ}&3^{\circ}&\mathbf{\ominus 3}\\
3^{\circ}&4^{\circ}&\mathbf{\ominus 4}\\
\ominus 3&\ominus 4&\mathbf{5}
\end{pmatrix} \Rightarrow v^{(3)} = \begin{pmatrix} \ominus 3\\ \ominus 4\\5\end{pmatrix}.\]
It is easy to see that $v^{(1)}\in (\smax^\vee)^{n}\setminus\{\zero\}$ and

\[
A v^{(1)}=\gamma_1 v^{(1)}=\begin{pmatrix}
9&\ominus 8&7
\end{pmatrix}^T.
\]
Therefore $v^{(1)}$ is a strong $\smax$-eigenvector. Also, $v^{(2)}$ and  $v^{(3)}$are $\smax$-eigenvectors since $v^{(2)}$ and $v^{(3)}\in (\smax^\vee)^{n}\setminus\{\zero\}$ and
\[
A v^{(2)}=\begin{pmatrix}
7^{\circ}&
\ominus 7&
\ominus 6
\end{pmatrix}^T \balance\;\gamma_2 v^{(2)}=\begin{pmatrix}
\ominus  6&
\ominus 7&
\ominus 6
\end{pmatrix}^T,
\]
and 
\[
A v^{(3)}=\begin{pmatrix}
6^{\circ}&
6^{\circ}&
6
\end{pmatrix}^T \balance \;\gamma_3 v^{(3)}=\begin{pmatrix}
\ominus 4&
\ominus 5&
6
\end{pmatrix}^T.
\]
They are not strong eigenvectors.
\end{example}

 \begin{example}\label{ex_eig} For another example, let 
$A = \begin{pmatrix}
3 &2& 1\\
2&2&1\\
1&1&1
\end{pmatrix}$.
As in the previous example, $A \in \pd_n(\smax^\vee)$ with the $\smax$-eigenvalues $\gamma_1=a_{11}=3$, $\gamma_2=a_{22}=2$ and $\gamma_3 = a_{33}= 1$. 
We obtain this time: 
\[  v^{(1)} = \begin{pmatrix} 6\\5\\4\end{pmatrix}\; ,\quad
v^{(2)} = \begin{pmatrix}  4\\\ominus 5\\\ominus 4\end{pmatrix}
\; , \quad 
v^{(3)} = \begin{pmatrix} 3^{\circ}\\\ominus 4\\5\end{pmatrix}.\]
It is easy to see that $v^{(1)}\in (\smax^\vee)^{n}\setminus\{\zero\}$ and
\[
A v^{(1)}=\gamma_1 v^{(1)}=\begin{pmatrix}
9&8&7
\end{pmatrix}^T.
\]
Therefore $v^{(1)}$ is a strong $\smax$-eigenvector. Also, $v^{(2)}$ is a $\smax$-eigenvector but not a strong one
since $v^{(2)}\in (\smax^\vee)^{n}\setminus\{\zero\}$ and
\[
A v^{(2)}=\begin{pmatrix}
7^{\circ}&
\ominus 7&
\ominus 6
\end{pmatrix}^T \neq \;\gamma_2 v^{(2)}=\begin{pmatrix}
 6&
\ominus 7&
\ominus 6
\end{pmatrix}^T,
\]
and $v^{(3)}$ is a weak $\smax$-eigenvector and not a  $\smax$-eigenvector since it has one balanced entries. 
\end{example}

 \begin{example}\label{ex_eig2} Let 
$A = \begin{pmatrix}
3 &\ominus 2& 0\\
\ominus 2&2&1\\
0&1&1
\end{pmatrix}
\in \pd_n(\smax^\vee)$ with again $\smax$-eigenvalues $\gamma_1=a_{11}=3$, $\gamma_2=a_{22}=2$ and $\gamma_3 = a_{33}= 1$. 
We have $Av^{(1)}=\gamma_1v^{(1)}$, but 
 \[ v^{(1)}=\begin{pmatrix}
6\\
\ominus 5\\
3^{\circ}
\end{pmatrix} \notin (\smax^{\vee})^n\setminus \{\zero\}\enspace .\]
By \Cref{coro-simple-eigen} we know that there is at least one $\smax$-eigenvectors of the form 
 $ \begin{pmatrix}
6\\
\ominus 5\\
3
\end{pmatrix}$ or
 $\begin{pmatrix}
6\\
\ominus 5\\
\ominus 3
\end{pmatrix}$.
In this example, both are $\smax$-eigenvectors.
\end{example}






\subsection{Computing the leading $\smax$-eigenvector using Kleene's star}
In  \Cref{coro-unique-eigen} we gave a condition under which a $\smax$-eigenvector
associated to a $\smax$-eigenvalue of a tropical positive definite matrix
is unique up to a multiplicative constant.
We shall give here another characterization of such a $\smax$-eigenvector
using Kleene'star of matrices, see \Cref{star_smax}.
We shall first consider the case when the eigenvalue is
the greatest one, in which case, we speak of a
\new{leading  $\smax$-eigenvector}. 

The following well known result is usually written using the maximal cycle mean which is equal to the maximal (algebraic or geometric) eigenvalue of $A$.
\begin{lemma}\label{leq_unit}
For $A \in (\tmax)^{n \times n}$, $A^*$ exists (in $\tmax$) if and only if 
all its eigenvalues are $\leq \unit$, and then $A^*= I \oplus A \oplus \cdots \oplus A^{ n-1}$.
\end{lemma}
The following result follows from idempotency of addition in $\smax$.
\begin{lemma}\label{eq_star}
For $A \in (\smax)^{n \times n}$ we have $
\tsum_{k=0,\ldots,m} A^{ k} = (I \oplus A)^{ m}$. \hfill \qed
\end{lemma}
\begin{lemma}\label{existence_star}
If $A \in (\smax)^{n \times n}$ and $|A|^*$ exists, then $A^{*} \in (\smax)^{n \times n}$ exists.
\end{lemma}
\begin{proof}
$\{\tsum_{k=0,\ldots,m} A^{ k}\}_m$ is a non-decreasing sequence with respect to $m$ for the order relation $\preceq$ (\Cref{partial_order}) and
its absolute value $|\tsum_{k=0,\ldots,m} A^{ k}|=  \tsum_{k=0,\ldots,m} |A|^{ k}$
which is stationary for $m\geq n$, and equal to $|A|^*$, by \Cref{leq_unit}.
So for $m \geq n$ the sequence is non-decreasing but can only take a finite number of values (the matrices $B$ such that $|B|=|A|^*$). Therefore, there exists $m_0\geq n$ such that $\tsum_{k=0,\ldots,m} A^{ k}$ is stationary for $m\geq m_0$.
\end{proof}

 We first state the main result of this section, which computes
the vector $v^{(1)}$
for a matrix  $A\in \pd_n(\smax^\vee)$ as in \Cref{balance-adj}, 
using Kleene's star of the matrix $\gamma^{-1} A$. 
\begin{theorem}\label{result_pro}
 Let $A\in \pd_n(\smax^\vee)$, $\gamma_k$ and $B_k$ be as in \Cref{balance-adj}.
Assume that $\gamma=\gamma_1$ is simple as an algebraic $\smax$-eigenvalue of $A$, that is $\gamma_1\succ \gamma_2$ 
Then, we have 
\[ v^{(1)}=(\gamma I \ominus A )^{\adj}_{:,1}=\gamma^{n-1} (\gamma^{-1}A)^*_{:,1}\enspace .\]
Moreover $A v^{(1)}= \gamma  v^{(1)}$.
In particular, when $v^{(1)} \in (\smax^\vee)^n$, $v^{(1)}$ is the unique
leading $\smax$-eigenvector, and this is a strong $\smax$-eigenvector.
\end{theorem}
We decompose the proof of \Cref{result_pro} into the following lemmas
which hold for $A$ as in the theorem.


\begin{lemma}\label{lemmaIB}
Let $\Azero$ be the matrix obtained by replacing the diagonal entries of $A$ by $\zero$. 
Then, we have $(I \ominus \gamma^{-1} A)^\adj_{:,1}=(I \ominus \gamma^{-1} \Azero)^{\adj}_{:,1}$.
\end{lemma}
\begin{proof}
We have $A = \Azero \oplus D$ where $D$ is the diagonal matrix with same diagonal entries as $A$, that is $D=(d_{ij})$ with $d_{ij}=\zero$ for $i \neq j$  and $d_{ij}=a_{ij}$ for $i=j$. The matrix $ I\ominus\gamma^{-1} D$ is a diagonal matrix with diagonal  entries equal to $\unit$ except for the first one which is 
equal to $\unit\ominus \gamma^{-1} d_{11}=\unit^\circ$.
So we get that
\[ 
( I \ominus \gamma^{-1} A)^{\adj}_{:,1}=( I \ominus \gamma^{-1} D \ominus \gamma^{-1}\Azero)^{\adj}_{:,1}\nonumber =(I \ominus \gamma^{-1}\Azero)^{\adj}_{:,1}\] 
where the last equality is by the fact that the entries of $I \ominus \gamma^{-1}\Azero$ and $I \ominus \gamma^{-1}D \ominus \gamma^{-1} \Azero$ only differ in the $(1,1)$ entry, which does not change the first column (and first row) of the adjugate matrix.
\end{proof}
\begin{definition}[Definite matrix]
A matrix $F=(f_{ij}) \in (\smax)^{n \times n}$ is definite if $\det(F)=f_{ii}=\unit\; \forall i \in [n]$.
\end{definition}
\begin{lemma}\label{lemma325}
Let $\Azero$ be as in \Cref{lemmaIB}. Then $I\ominus \gamma^{-1} \Azero$ is definite.
\end{lemma}
\begin{proof}
Let $F=(f_{ij}):=I\ominus \gamma^{-1} \Azero$.
We have $f_{ii}=\unit\; \forall i \in [n]$, 
and $f_{ij}=\ominus \gamma^{-1} a_{ij}$ for $i\neq j\in [n]$.
Since $A$ is a $\pd$ matrix with all diagonal entries $\preceq \gamma$, we get that all off-diagonal entries of $\gamma^{-1} A $.
 Therefore all weights of cycles of  $\gamma^{-1} \Azero$ have an
absolute value $\prec \unit$.
This implies that $\det(F)=\unit$ and so $F$ is a definite matrix.
\end{proof}
\begin{lemma}[Corollary of \protect{\cite[Th.\ 2.39]{adi}}]
\label{adj_star1}
Let $A$ and $\Azero$ be as in \Cref{lemmaIB}.
Then, $(\gamma^{-1} \Azero)^*$ exists and we have $(\gamma I \ominus \Azero)^{\adj}=\gamma^{n-1} ( \gamma^{-1}\Azero)^*$.
\end{lemma}
\begin{proof}
As said in the previous proof, all weights of cycles of  $\gamma^{-1} \Azero$ have an absolute value $\prec \unit$.
Since  $ \gamma^{-1} \Azero$ has diagonal entries equal to $\zero$, and $I\ominus \gamma^{-1} \Azero$ is definite by previous lemma, we get
the assertions in \Cref{adj_star1}, by applying \cite[Th.\ 2.39]{adi}. 
\end{proof}
\begin{lemma}\label{star_star1}
Let $A$ and $\Azero$ be as in \Cref{lemmaIB}. Then $ ( \gamma^{-1}\Azero)^*=
(\gamma^{-1}A)^*$.
\end{lemma}
\begin{proof}
Without loss of generality, we prove the result when $\gamma=\unit$.
Using \Cref{leq_unit} together with \Cref{existence_star}, we get that $A^*$ exist. Suppose that the sequences $\{\tsum_{k\in [m]} {A^k}\}_m$ and $\{\tsum_{k\in [m]} \Azero^k\}_m$ be stationary for $m\geq m_1$ and $m \geq m_2$, respectively. Let $m'=\max\{m_1,m_2\}$. Then we have
\[\Azero^*=\bigtsum_{k=0,\ldots, m'}\Azero^k = (I \oplus \Azero)^{m'} = (I \oplus A)^{m'} = \bigtsum_{k=0,\ldots,m'} A^k=A^*,\]
where the second and fourth equalities are by \Cref{eq_star}, and the third one is by definition of $\Azero$.
\end{proof}

\begin{proof}[Proof of \Cref{result_pro}]
Consider $A'=\gamma^{-1}A$. 
Using the multi-linearity of determinant, or using 
\cite[Cor.\ 2.35]{adi},
we get
\[v^{(1)}=(\gamma I\ominus A )^{\adj}_{:,1}=\gamma^{n-1} (I\ominus  A')^{\adj}_{:,1}\]
and using \Cref{lemmaIB},  \Cref{adj_star1} and \Cref{star_star1},  we get the respective equalities 
\[ (I\ominus A')^{\adj}_{:,1}= (I \ominus  \gamma^{-1}\Azero)^{\adj}_{:,1}= (\gamma^{-1}\Azero)^*_{:,1}=( \gamma^{-1} A)^*_{:,1}= (A')^*_{:,1}.\]
This shows the first assertion of \Cref{result_pro}.
Since  $(A')^*=I\oplus A'(A')^*$ and  $[A'(A')^*]_{11}\succeq A'_{11}=\unit$,
we get that $ (A')^*_{11}=\unit \oplus [A'(A')^*]_{11}=[A'(A')^*]_{11}$,
and so $A'(A')^*_{:,1}=(A')^*_{:,1}$, which with the first assertion,
shows the second assertion  of \Cref{result_pro}.
The last assertion follows from \Cref{coro-unique-eigen}.
\end{proof}
\begin{corollary}
Let  $A$ be as in  \Cref{result_pro}. Assume that all the entries 
of $A$ are positive or $\zero$, that is are in $\smax^{\oplus}$.
Then, $v^{(1)}$ has also positive or $\zero$ entries, and thus it is 
necessarily a strong $\smax$-eigenvector.
\end{corollary}

\begin{remark}
The previous result shows that for a positive definite matrix with
positive or zero entries, the leading $\smax$-eigenvector of $A$ is positive.
This is not the case in general for the other eigenvectors. For
instance, in 
\Cref{ex_eig}, $v^{(2)}$ contains positive and negative coordinates, 
and since it is in $(\smax^\vee)^{n}$, this is the unique eigenvector associated to $\gamma_2$. Moreover, $v^{(3)}$ contains balance entries.
\end{remark}
\begin{remark}\label{not_signed}
Consider the matrix $A$ in \Cref{ex_eig2} and let $\gamma=3$. 
We have
\[(\gamma^{-1}A)^*=(\gamma^{-1}\Azero)^*=\begin{pmatrix}0 & \ominus (-1)&(-3)^\circ \\\ominus (-1) &0 &-2 \\  (-3)^\circ &  -2 &0 \end{pmatrix}\enspace.\]
So by \Cref{result_pro}, $v^{(1)}=\gamma^2 (\gamma^{-1}A)^*_{:,1}=
6 \odot \begin{pmatrix} 0 & \ominus (-1) &  (-3)^\circ  \end{pmatrix}^\intercal$.
\Cref{result_pro} also shows that $A v^{(1)}= \gamma v^{(1)}$.  However, $v^{(1)}$
is not in $(\smax^\vee)^n$, and the two possible $\smax$-eigenvectors
mentioned in \Cref{ex_eig2} are not strong  $\smax$-eigenvectors.
\end{remark}
The later remark is indeed a general result, as follows.
\begin{corollary}\label{coro-strong1}
Let  $A$ and $\gamma$ be as in  \Cref{result_pro}. 
If $v^{(1)}$ does not belong to $(\smax^\vee)^n$, then $A$ has no strong $\smax$-eigenvector associated to the eigenvalue $\gamma$.
\end{corollary}
\begin{proof}
Assume that $w$ is a strong $\smax$-eigenvector associated to the eigenvalue $\gamma$.
This means that $w\in (\smax^\vee)^n\setminus\{\zero\}$ and
$\gamma w=A w$. This implies 
$w= (\gamma^{-1}A)^* w$ and so $w\succeq  (\gamma^{-1}A)^*_{:,1} w_1$.
Since $w$ is necessarily a $\smax$-eigenvector, \Cref{coro-simple-eigen}
shows that there exists $\lambda \in \smax^\vee\setminus\{\zero\}$ such that
$w\balance \lambda v^{(1)}$.
We also have $v^{(1)}_1\in \smax^\vee\setminus\{\zero\}$, 
so  $w_1=\lambda v^{(1)}_1$, by \Cref{equality_balance}.
By \Cref{result_pro}, we have 
$v^{(1)}=\gamma^{n-1}(\gamma^{-1}A)^*_{:,1}= v^{(1)}_1(\gamma^{-1}A)^*_{:,1}$, therefore 
$ \lambda v^{(1)}= w_1 (\gamma^{-1}A)^*_{:,1}\preceq w$.
Since $w\balance \lambda v^{(1)}$, and  $w\in (\smax^\vee)^n\setminus\{\zero\}$
this implies, by \Cref{equality_balance}, that $w=\lambda v^{(1)}$ and so $v^{(1)}\in (\smax^\vee)^n\setminus\{\zero\}$, a contradiction.
\end{proof}
\subsection{Computing all $\smax$-eigenvectors using Kleene's star}\label{subsec:kleen}
\begin{theorem}\label{star-general}
Let $A$, $\gamma_k$ and $B_k$ be as in \Cref{balance-adj}.
Assume that $\gamma=\gamma_k$ is simple as an algebraic $\smax$-eigenvalue of $A$.
Let $D^{(k)}$ be the $n\times n$ diagonal matrix with diagonal entries
$\gamma_1,\ldots, \gamma_{k-1}, \zero,\ldots, \zero $,  
let $A^{(k)}$ be the complementary of $D^{(k)}$ in $A$, that is $A^{(k)}$ 
has entries $[A^{(k)}]_{ij}= a_{ij}$ when $i\neq j$ or $i=j\geq k$, and $\zero$
elsewhere, so that $A= D^{(k)}\oplus  A^{(k)}$.
Then, we have 
\[ v^{(k)}=(B_k)^\adj_{:,k}= \lambda_k ((\gamma I\ominus D^{(k)})^{-1} A^{(k)})_{:,k}^*\quad \text{with}\; \lambda_k= (\ominus \unit)^{k-1} \gamma_1\cdots \gamma_{k-1}\gamma ^{n-k} \enspace ,\]
and $v^{(k)}$ satisfies $A^{(k)} v^{(k)}=\gamma v^{(k)}\ominus D^{(k)} v^{(k)}$.

Moreover, if $v^{(k)}\in (\smax^\vee)^n$, $v^{(k)}$ is the unique
$\smax$-eigenvector associated to the eigenvalue $\gamma$ (up to a multiplicative factor). 
\end{theorem}

\begin{proof}
Let $A$, $\gamma_k$, $B_k$, $D^{(k)}$ and $A^{(k)}$ be as in the theorem.
We have $A= D^{(k)}\oplus  A^{(k)}$,
so $B_k= \gamma I\ominus A= \gamma I\ominus D^{(k)}\ominus  A^{(k)}$.
Since $\gamma_1\succeq \cdots \succeq \gamma_{k-1}\succ \gamma_k=\gamma
\succ \gamma_{k+1}\cdots \gamma_n$, the matrix 
$\gamma I\ominus D^{(k)}$ is diagonal with diagonal entries
equal to $ \ominus \gamma_1, \ldots, \ominus \gamma_{k-1},\gamma, \cdots 
\gamma$, so is invertible.
Hence, 
\begin{align*}
 (B_k)^\adj&=(I\ominus (\gamma I\ominus D^{(k)})^{-1} A^{(k)})^\adj  (\gamma I\ominus D^{(k)})^\adj\\
&= (I\ominus (\gamma I\ominus D^{(k)})^{-1} A^{(k)})^\adj 
\det(\gamma I\ominus D^{(k)}) (\gamma I\ominus D^{(k)})^{-1}
\enspace .\end{align*}
Then, setting $\lambda_k= (\ominus \unit)^{k-1} \gamma_1\cdots \gamma_{k-1}\gamma ^{n-k}$, we get that the $k$-th column satisfies
$v^{(k)}=(B_k)^\adj_{:,k}= \lambda_k  (I\ominus (\gamma I\ominus D^{(k)})^{-1} A^{(k)})^\adj_{:,k}$. Let $A'= (\gamma I\ominus D^{(k)})^{-1} A^{(k)}$. Then, the formula of $v^{(k)}$ in the theorem holds as soon as 
$ (A')^*_{:,k}= (I\ominus A')^\adj_{:,k}$.

One shows this property by the same arguments as for the proof of \Cref{result_pro}. 

First, writting $A'=\Azero'\oplus D$ where $D$ is the diagonal matrix with
same diagonal entries as $A'$, 
and $\Azero'=(\gamma I\ominus D^{(k)})^{-1} \Azero$,
we see that $I\ominus D$ is a diagonal matrix 
with entries equal to $\unit$ except for the $k$-th entry which is equal to $\unit^\circ$. So by the same arguments as in \Cref{lemmaIB},
we get $(I\ominus A')^\adj_{:,k}= (I\ominus \Azero')^\adj_{:,k}$.

Second, although some off-diagonal entries of $A'$ may have an absolute value 
$\succeq \unit$, all weights of cycles of length $\geq 2$
have an absolute value $\prec \unit$.
So, as in the proof of \Cref{lemma325}, we have that $I\ominus A'$ is
definite. 

Then, the arguments of \Cref{adj_star1} and \Cref{star_star1} apply 
when replacing the matrix $\gamma^{-1} A$ by $A'$ and $\gamma^{-1} \Azero$
by $\Azero'$, and shows that
$(\Azero')^*$ exists and $(I \ominus \Azero')^{\adj}=(\Azero')^*$
and $ (\Azero')^*=(A')^*$.

All together, this implies that $ (A')^*_{:,k}=  (\Azero')^*_{:,k}=(I \ominus \Azero')^{\adj}_{:,k}=(I\ominus A')^\adj_{:,k}$, which is what we wanted to prove, and finishes the proof of the first assertion.

For the last assertion, we proceed again as in the proof of \Cref{result_pro}. 
Since  $(A')^*=I\oplus A'(A')^*$ and  $[A'(A')^*]_{kk}\succeq A'_{kk}=\unit$,
we get that $ (A')^*_{kk}=\unit \oplus [A'(A')^*]_{kk}=[A'(A')^*]_{kk}$,
and so $A'(A')^*_{:,k}=(A')^*_{:,k}$, which with the first assertion,
shows the second assertion  of \Cref{star-general}.

The last assertion follows from \Cref{coro-unique-eigen}.
\end{proof}

The construction of the system of
equations $A^{(k)} v^{(k)}=\gamma v^{(k)}\ominus D^{(k)} v^{(k)}$
in \Cref{star-general} can be compared to the ones proposed in
\cite{Nishida2020,Nishida2021,nishida2021independence}
for the definition of ``algebraic eigenvectors''. The 
novelty here is to consider {\em signed} tropical matrices.
Moreover, considering symmetric positive definite matrices
leads to a special form of this construction: the ``multi-circuits'' of 
the above references are only 
unions of loops, so that the equation involves 
a diagonal matrix.

We can also show a similar result as \Cref{coro-strong1} for all the eigenvalues.
\begin{corollary}Let  $A$ and $\gamma$ be as in  \Cref{star-general}.
If $v^{(k)}$ does not belong to $(\smax^\vee)^n$, or if $A$ is irreducible, then 
$A$ has no strong $\smax$-eigenvector associated to the eigenvalue $\gamma$.
\end{corollary}
\begin{proof}
Assume by contradiction that $w$ is a strong $\smax$-eigenvector associated to the eigenvalue $\gamma$. 
This means that $w\in (\smax^\vee)^n\setminus\{\zero\}$ and
$\gamma w=A w$. Using the decomposition $A=A^{(k)}+D^{(k)}$ of 
\Cref{star-general}, we see that $\gamma_k w\succeq D^{(k)} w$,
hence $\gamma_k w_i\succeq \gamma_i w_i$, for $i=1,\ldots, k-1$.
Since $\gamma_i>\succ \gamma_k$, this implies that $w_i=\zero$, for all
$i=1,\ldots, k-1$.
Then $D^{(k)} w=\zero$ and we get $A^{(k)} w= Aw= \gamma w= (\gamma I\ominus D^{(k)})w$. 
 This implies 
$w= ((\gamma\ominus  D^{(k)})^{-1}A)^* w$.
Using the same arguments as in the proof of \Cref{coro-strong1}, 
with  \Cref{star-general} instead of \Cref{result_pro}, we deduce
that $w=\lambda v^{(k)}$ for some $\lambda \in \smax^\vee\setminus\{\zero\}$.
This implies that $v^{(k)}\in (\smax^\vee)^n\setminus\{\zero\}$
and that the first entries $v^{(k)}_i$ with $i=1,\ldots, k-1$ are equal to $\zero$.
If $A$ is irreducible, then so does $(\gamma\ominus  D^{(k)})^{-1}A$,
and by \Cref{irreducible},
$((\gamma\ominus  D^{(k)})^{-1}A)^*$ has only non-zero entries.
This implies that the vector $v^{(k)}=\lambda_k ((\gamma\ominus  D^{(k)})^{-1}A)^*_{:,k}$ has no zero entries, a contradiction.
\end{proof}

\subsection{Generic uniqueness of $\smax$-eigenvalues and $\smax$-eigenvectors}
\label{sec-generic}
\Cref{ex_eig2} shows that a $\smax$-eigenvector is not necessarily unique up to a multiplicative constant, even when the corresponding $\smax$-eigenvalue is simple and even when this eigenvalue is the greatest one.

Although there are many examples as above, we also have the following result, which shows that generically the $\smax$-eigenvalues are unique and the $\smax$-eigenvectors are unique up to a multiplicative constant.

To any matrix $A=(a_{ij})_{i,j\in [n]} \in \pd_n(\smax^{\vee})$, we associate
the vector $\Psi(A)=(a_{ij})_{1\leq i\leq j\leq n}$ of $(\smax^{\vee})^{n(n+1)/2}$.
The map $\Psi$ is one to one and we denote by $\upd$ the 
image of $\pd_n(\smax^{\vee})$ by the map $\Psi$.
Similarly to univariate polynomials, see \Cref{sec-polynomials},
multivariate formal polynomials and polynomial functions can be defined in any semiring ${\mathcal S}$.
We can also consider formal Laurent polynomials, which are sequences
$P_\alpha\in {\mathcal S}$ indexed by elements $\alpha\in \Z^d$ (where $d$ is the number of variables and $\Z$ is the set of relative integers),
such that $P_\alpha=\zero$ for all but a finite number of indices $\alpha$.
Then the Laurent polynomial function $\widehat{P}$
can only be applied to vectors $x\in ({\mathcal S}\setminus\{\zero\})^d$.

The following result shows that generically,
a matrix $A \in \pd_n(\smax^{\vee})$ has simple $\smax$-eigenvalues and unique 
$\smax$-eigenvectors associated all its eigenvalues, up to multiplicative constants, and these  $\smax$-eigenvectors are the vectors  $v^{(k)}$.
\begin{theorem} \label{theoremgeneric}
There exist a finite number of 
formal polynomials $(P_k)_{k\in I}$  over $\smax$, with coefficients in $\smax^\vee$ and  $n(n+1)/2$ variables, 
such that for 
any matrix $A \in \pd_n(\smax^{\vee})$, and $x=\Psi(A)$, 
satisfying $\widehat{P_k}(x)\in \smax^{\vee}\setminus\{\zero\}$ for all $k\in I$, we have that 
$A$ has simple $\smax$-eigenvalues and all the vectors $v^{(k)}$ defined
as above are in $(\smax^{\vee}\setminus\{\zero\})^n$.
\end{theorem}
\begin{proof}
Let $x=\Psi(A)$ for some matrix $A \in \pd_n(\smax^{\vee})$, then
$x=(x_{ij})_{1\leq i\leq j\leq n}\in (\smax^\vee)^{n(n+1)/2}$,
$x_{ii}\in \smax^{\oplus}$, for all $i\in [n]$ and $x_{ij}< x_{ii}\odot x_{jj}$
for all $1\leq i<j\leq n$.
The $\smax$-eigenvalues of $A$ are simple if and only if
$x_{ii}\neq x_{jj}$ for all $i\neq j$.
Take for all $i\neq j$, the formal polynomial $P_{ij}$ equal to $x_{ii}\ominus x_{jj}$, we get that the latter holds if and only if
$\widehat{P_{ij}}(x)$ is not balanced for all $i\neq j$.

If this holds, then there is a permutation of the indices
$i\in [n]$, such that $x$ satisfies  $x_{11}>\cdots>x_{nn}$.
Then, the vectors $v^{(k)}$ are given 
by the formula in \Cref{star-general}, that is, for all $i\neq k$, we have 
$v^{(k)}_i= \lambda_k  ((\gamma_k I\ominus D^{(k)})^{-1} A^{(k)})_{i,k}^*$
where
$\gamma_k=a_{kk}=x_{kk}$,
$\lambda_k=v^{(k)}_k=  (\ominus \unit)^{k-1} \gamma_1\cdots \gamma_{k-1}\gamma_k^{n-k}$,
$\gamma_k I\ominus D^{(k)}$ is the diagonal matrix with diagonal entries
$\ominus \gamma_1,\ldots, \ominus \gamma_{k-1},\gamma_k,\ldots, \gamma_k$,
or equivalently $\ominus x_{11},\ldots,\ominus x_{k-1,k-1}, x_{kk},\ldots, x_{kk}$,
and $A^{(k)}$ 
has entries $A^{(k)}_{ij}= a_{ij}=x_{ij}$ when $i<j$, or $i=j\geq k$, 
$A^{(k)}_{ij}= a_{ij}=x_{ji}$ when $j<i$, and $\zero$ elsewhere.
Using that the above Kleene's star exists and is equal to the finite sum
of the powers of the matrix up to power $n$, we deduce that the 
expressions of all the $v^{(k)}_i$ are Laurent polynomials functions 
in the entries of $x$.
These polynomials correspond to the weights of
paths from $i$ to $k$ in the graph of $A$,
defined by the matrix
$(\gamma_k I\ominus D^{(k)})^{-1} A^{(k)}$,
and since we assumed that the $\smax$-eigenvalues are simple, 
the above matrix has no circuit of weight $0$ (see the 
proof of \Cref{star-general}), then we can restrict the sum to elementary
paths (that is paths with no circuits).
Let $Q_{ki}$ be the formal (Laurent) polynomial corresponding to the
expression of $v^{(k)}_i$, 
then the above properties imply that each monomial in the coordinates of $x$ 
appears only once, and so the coefficients of all monomials are in $\smax^\vee$.
Moreover, due to the multiplication of the weights 
by $\lambda_k$, the degree of each monomial with respect 
to each variable is $\geq 0$ that is $Q_{ki}$ is an ordinary polynomial.
Then, all the $v^{(k)}$ are in $(\smax^\vee)^n$ if and only if
$\widehat{Q_{ki}}(x)\in \smax^\vee$ for all $k\neq i$.
This property holds indeed if  $x_{11}>\cdots>x_{nn}$.
Now take apply any permutation $\sigma$ of the indices on the polynomials 
$Q_{ki}$, we obtain polynomials $Q^{\sigma}_{ki}$.
The previous properties imply that if
$\widehat{P_{ij}}(x)$ is not balanced for all $i\neq j$
and $\widehat{Q^{\sigma}_{ki}}(x)\in \smax^\vee$ for all $k\neq i$ and
$\sigma$ a permutation of $[n]$, then
$A$ has simple eigenvalues and all the vectors $v^{(k)}$ belong 
to $(\smax^\vee\setminus\{\zero\})^n$.
\end{proof}

Another way to see this result is to consider polyhedral complexes,
see for instance in \cite{maclagan2015introduction}.
A polyhedral complex $\mathcal C$ in $\R^m$, $m\geq 1$,
 is a collection of polyhedra
of $\R^m$, called cells, such that the intersection of any two polyhedra 
is a face of each of the two polyhedra or is empty.
Then, the support of $\mathcal C$ is the union of all its cells,
and its dimension is the maximal dimension $p\leq m$ of its cells
(and its codimension is $m-p$). If the support is convex, then the dimension
of the complex is the dimension of its support.
The set of cells of dimension at most $p-1$ is a sub-polyhedral complex of
$\mathcal C$ with dimension $p-1$, which is the complementary
in the support of $\mathcal C$ of the union of the interiors of the
cells of $\mathcal C$ of maximal dimension.
Polyhedral complexes are useful to explain the geometric structure of
tropical varieties or prevarieties.
Given a $m$-variables formal (Laurent) polynomial $P$ over $\rmax$, the set of 
vectors $x\in \R^m$ such that the maximum in the expression of
$\widehat{P}(x)$ is attained at least twice defines a tropical hypersurface.
It defines a  polyhedral complex $\mathcal C$ in $\R^m$ covering $\R^m$,
such that tropical hypersurface is exactly the support of 
sub-polyhedral complex of $\mathcal C$ obtained by taking the
cells of dimension at most $m-1$.
In particular, the set of points for which the maximum  in the expression of
$\widehat{P}(x)$ is attained only once is the union of the interiors 
of cells of maximal dimension of $\mathcal C$.

Now assume that $\Gamma=\R$.
 We denote by $\mu$ the map from $(\smax^\vee)^m$ to $\rmax^m$
which associates to any vector $x$, the vector $|x|$ 
obtained by applying the modulus map entrywise. Then, any element of $\rmax^m$
has a finite preimage by $\mu$ of cardinality at most $2^m$.
Moreover, if $P$ is a formal polynomial with coefficients in $\smax^\vee$ and
 $m$ variables, and 
$Q$ denotes the formal polynomial with coefficients $Q_\alpha=|P_{\alpha}|$,
then the set of points  $x\in(\smax^\vee\setminus\{\zero\})^m$
such that $\widehat{P}(x)\in \smax^\circ$ is included in the preimage 
by $\mu$ of the set of points $x\in \R^m$ such that the maximum in
the expression of $\widehat{Q}(x)$ is attained at least twice.

Consider the set 
$K$ which is the image by $\Phi:=\mu\circ \Psi$ of $\pd_n(\smax^{\vee})$.
We have
\[ K\cap \R^{n(n+1)/2}=\{x=(x_{ij})_{1\leq i\leq j\leq n}\in \R^{n(n+1)/2}\mid
2 x_{ij}< x_{ii}+x_{jj}\;\forall\; 1\leq i<j\leq n\}\enspace ,\]
where we used the usual notations of $\R$.
Therefore, the closure of $K\cap \R^{n(n+1)/2}$  in $\R^{n(n+1)/2}$ 
is a convex polyhedron of dimension
$n(n+1)/2$ that we shall denote by $\widetilde{K}$.
The previous result implies the following one.

\begin{corollary}\label{corogeneric}
There exists a  polyhedral complex $\mathcal C$ of $\R^{n(n+1)/2}$ with support 
$\widetilde{K}$, such that the image by $\Phi$ of the set of matrices
$A \in \pd_n(\smax^{\vee})$ such that 
all the eigenvalues of $A$ are simple and all the vectors $v^{(k)}$ defined
as above are in $\smax^{\vee}\setminus\{\zero\}$
contains all the interiors of cells of maximal
dimension of $\mathcal C$.
\end{corollary}
\begin{proof}
By \Cref{theoremgeneric}, we have that if $x=\Psi(A)$ 
with $A \in \pd_n(\smax^{\vee})$, and
$\widehat{P_k}(x)\in \smax^\vee\setminus\{\zero\}$ for all $k\in I$, then 
$A$ has simple $\smax$-eigenvalues and all the vectors $v^{(k)}$ defined
as above are in $(\smax^{\vee}\setminus\{\zero\})^n$.
Consider the formal polynomial $Q_k=|P_k|$, then 
by the above comments, $\widehat{P_k}(x)\in \smax^\vee\setminus\{\zero\}$
holds if $|x|\in \R^m$  and the maximum in
the expression of $\widehat{Q_k}(|x|)$ is attained only once.
Each polynomial $Q_k$ defines a polyhedral complex ${\mathcal C}_k$
covering $\R^{n(n+1)/2}$, such that
the set of points for which the maximum  in the expression of
$\widehat{Q_k}(x)$ is attained only once is the union of the interiors 
of cells of maximal dimension of $\mathcal C_k$.
Let $\mathcal C$ be 
the refinement of all the  polyhedral complexes ${\mathcal C}_k$.
We get that if $|x|=\Phi(A)$ is in the interior of a cell
 of maximal dimension of $\mathcal C$, then $A$ 
has simple $\smax$-eigenvalues and all the vectors $v^{(k)}$ defined
as above are in $(\smax^{\vee}\setminus\{\zero\})^n$.

Taking the intersection of the cells with $\widetilde{K}$,
we obtain a polyhedral complex with support equal to   $\widetilde{K}$,
and satisfying the property in the corollary. 
\end{proof}

\section{Applications}\label{sec:apps}
\subsection{Signed valuations of positive definite matrices}
\begin{definition}[See also \protect{\cite{allamigeon2020tropical}}]
\label{def-sign}
For any ordered field $\rfield$, we define the {\em sign map}
$\sign:\rfield \to \bmaxs^\vee=\{\zero, \ominus \unit, \unit\}$ as follows, for all $\elf\in \rfield$,
\[ \sign(\elf):=\begin{cases}
\unit&\elf>0,\\
\ominus \unit& \elf<0,\\
\zero& \elf=0. 
\end{cases}\]
Assume that $\rfield$ is also a valued field, with a convex valuation  $\vall$ with value group $\vgroup$, then the \textit{signed valuation} on $\rfield$ is the map $\sval:\rfield \to \smax=\smax(\vgroup)$ which associates the element $\sign(\elf) \odot \vall(\elf)\in \smax^\vee$, with an element $\elf$ of $\rfield$,
where here $\vall(\elf)$ is seen as an element of $\smax^{\oplus}$.
\end{definition}
Recall that the value group is an ordered group such that $\vall \colon \rfield \to \vgroup \cup \{\botelt \}$ is surjective, that $\vall$ is a (non-Archimedean) valuation if 
\[\begin{aligned}
 \vall(b) = \botelt &\iff b = 0 \, , \\
\forall b_{1}, b_{2} \in \rfield, \ \vall(b_{1}b_{2}) &= \vall(b_{1}) + \vall(b_{2}) \, , \\
\forall b_{1}, b_{2} \in \rfield, \ \vall(b_{1} + b_{2}) &\le \max(\vall(b_{1}),\vall(b_{2})) \, 
\end{aligned}\]
and that the convexity of $\vall$ is equivalent to the condition that
$\elf_{1} , \; \elf_{2} \in \rfield\, ,\;\text{and}\; 
 0 \le \elf_{2} \le \elf_{1}$ implies $\vall(\elf_{2})\leq  \vall(\elf_{1})$.
Then, the signed valuation is surjective, and it is 
a morphism of hyperfields or of systems, meaning
in the case of a map from a field $\rfield$ to $\smax$, that it satisfies,
for all $a,a',b\in \rfield$, we have 
\begin{align*}
&\sval(0)=\zero \\ 
& \sval(\rfield\setminus\{0\})\subset \smax^\vee\setminus\{\zero\}\\
&\sval(-a)=\ominus \sval(a)\\
&\sval(a + a')\preceq^\circ \sval(a)\oplus \sval(a')\\
&\sval(b a)=\sval(b)\odot \sval(a).
\end{align*}
The above construction applies in particular when $\rfield$ is a
real closed valued field with a convex valuation, and so 
to the field of formal generalized Puiseux series.

Signed valuation can be applied to matrices over $\rfield$ and polynomials 
over $\rfield$.
Then, it is proved in \cite[Th.\ 4.8]{tropicalization} that any tropical positive semidefinite matrix is the image by the signed valuation of a positive semidefinite matrix over the field of formal generalized Puiseux series.
We can obtain the same 
property for positive definite matrices, and a general field $\rfield$.

\begin{proposition}[Compare with \protect{\cite[Th.\ 4.8]{tropicalization}}]
Let $\rfield$ be an ordered field with a convex valuation
with value group $\vgroup$, and let $A\in \pd_n(\smax^{\vee})$.
Then, there exists a $n\times n$ symmetric  positive definite matrix ${\bf A}$ 
over $\rfield$ such that  $\sval({\bf A})= A$.
\end{proposition}
\begin{proof}
By the surjectivity of the signed valuation, there exists a $n\times n$
symmetric  matrix ${\bf A}$ 
over $\rfield$ such that  $\sval({\bf A})= A$.
Since $A$ is positive definite in $\smax$, 
\Cref{diag_cycle} shows that
for any cycle  $\cycle$ of length $k>1$ in $[n]$,
and support $I$, 
we have $|w(\cycle)| \lsign \bigtprod_{i\in I}a_{ii}$.
Then, the $I\times I$ submatrices of $A$ and ${\bf A}$ satisfy
$\sval(\det({\bf A}[I,I]))=\det(A[I,I])=\bigtprod_{i\in I}a_{ii}\in \smax^{\oplus}\setminus \{\zero\}$.
By definition of sign valuation, this implies that 
$\det({\bf A}[I,I]))>0$. Since this holds for all subsets $I$ of $[n]$
with at least two elements, and the same holds trivially for singletons $I=\{i\}$, we obtain that ${\bf A}$ has all its principal minors positive and so is positive definite in $\rfield$.
\end{proof}
\begin{theorem}\label{th_asymptot}
Let $\rfield$ be a real closed valued field with convex valuation
and value group $\vgroup$, and let 
$A\in \pd_n(\smax^{\vee})$.
Let ${\bf A}$ be a symmetric  positive definite matrix over $\rfield$
such that  $\sval({\bf A})= A$.
Then, the signed valuations of the eigenvalues of ${\bf A}$, counted with multiplicities, coincide with the algebraic $\smax$-eigenvalues of $A$,
that is the diagonal entries of $A$.
 \end{theorem}
 \begin{proof}
Since ${\bf A}\in \rfield^{n\times n}$, so the characteristic polynomial of ${\bf A}$, ${\bf Q}:= \det( \X I- {\bf A} )$ belongs to $\rfield[\X]$
and its coefficients are equal to
${\bf Q}_k =(-1)^{n-k}\tr_{n-k}({\bf A})$.
Using the above morphism property of the signed valuation,
we have $\sval({\bf Q}_k) \preceq^\circ (\ominus \unit)^{n-k}\tr_{n-k}(A)$.
Since $A$ is symmetric positive definite $(\ominus \unit)^{n-k}\tr_{n-k}(A)\in 
\smax^\vee$, so we get the equality.
This shows that the signed valuation of the characteristic polynomial ${\bf Q}$
of ${\bf A}$ is equal to the  characteristic polynomial $P_A=\det( \X I\ominus A ) $ of $A$ over $\smax$.
Since any symmetric matrix over a real closed field has $n$ eigenvalues counted
with multiplicites, ${\bf Q}$ has $n$ roots.
Using \cite[Prop.~B]{baker2018descartes} (see also \cite[Cor.\ 7.2]{tavakolipour2021}), we obtain that $P_A=\sval({\bf Q})$ has $n$ roots in $\smax$,
which coincide with the signed valuations of the roots of ${\bf Q}$,
counted with multiplicity. This means that the signed valuations of the eigenvalues of ${\bf A}$, counted with multiplicities, coincide with the algebraic $\smax$-eigenvalues of $A$.
\end{proof}

\Cref{th_asymptot} may be thought of as a non-Archimedean variation
of Gershgorin's disk theorem. The latter shows that the
eigenvalues of a matrix are not too far from is diagonal
entries. This analogy is better explained by means
of the following result, which differs from Gershgorin's
theorem in that we use a ``modulus of tropical positive
definiteness'' $\gamma$ instead of a notion of diagonal
dominance.

\begin{theorem}
Suppose $A$ is a real symmetric matrix with positive diagonal entries such that $a_{ii}a_{jj}\geq \gamma^2 a_{ij}^2$ for some $\gamma>0$. Then,
\[
\operatorname{spec}(A) \subset \cup_i B(a_{ii}, a_{ii} (n-1)/\gamma) 
\]
\end{theorem}
\begin{proof}
  Set $D:= \operatorname{diag}(a_{ii}^{-1/2})$, and consider
  $B=D (A-\gamma I) D$. Observe that $|B_{ij}|\leq \gamma^{-1}$
  for $i\neq j$ whereas $B_{ii}=1-\gamma/a_{ii}$. If $\gamma$
  is an eigenvalue of $A$, then the matrix $B$ 
  is singular, hence, it cannot have a dominant diagonal,
  which entails there is an index $i\in[n]$ such that
  $|1-\gamma/a_{ii}|\leq (n-1) \gamma^{-1}$.
  It follows that $\gamma \in B(a_{ii},a_{ii}(n-1)\gamma^{-1})$. 
  \end{proof}


Now let us state a result concerning eigenvectors.
\begin{theorem}\label{th_asymptot-vector}
Let $\rfield$ be a real closed valued field with convex valuation
and value group $\vgroup$.
Let $A$, $\gamma_k$ and $B_k$ be as in \Cref{balance-adj}.
Let ${\bf A}$ be a symmetric  positive definite matrix over $\rfield$
such that  $\sval({\bf A})= A$.
Let $\lambda_1\geq \cdots \geq \lambda_n>0$ be the eigenvalues 
of ${\bf A}$.

Assume that $\gamma=\gamma_k$ is simple as an algebraic $\smax$-eigenvalue of $A$. Then, $\lambda_k$ is a simple eigenvalue.
Let ${\bf v}$ be the eigenvector of ${\bf A}$ associated to the eigenvalue $\lambda_k$ such that 
${\bf v}_k=1$.
Then
$\sval({\bf v})\balance (v^{(k)}_k)^{-1} v^{(k)}$.

Assume in addition that   $v^{(k)}=(B_k)^\adj_{:,k}\in (\smax^\vee)^n$. Then,
$\sval({\bf v})= (v^{(k)}_k)^{-1} v^{(k)}$.
 \end{theorem}
\begin{proof}
We have ${\bf A} {\bf v}=\lambda_k {\bf v}$.
Taking the signed valuation, we have
$\gamma \odot \sval( {\bf v})=
\sval(\lambda_k) \sval( {\bf v})= \sval({\bf A} {\bf v})\preceq^\circ 
A \sval( {\bf v})$.
So $\sval( {\bf v})$ is a $\smax$-eigenvector of $A$.
From  Point (iii) of \Cref{coro-simple-eigen}, we
get that
$\sval( {\bf v})\balance \lambda v^{(k)}$ for some $\lambda\in \smax^{\vee}\setminus\{\zero\}$.
This implies in particular that $v_i= \lambda v^{(k)}_i$  for all $i$ such that
$v^{(k)}_i\in \smax^\vee$, and so for $i=k$.
Since ${\bf v}_k=1$, we get that $\sval( {\bf v})_k=\unit= \lambda v^{(k)}_k$,
so $\lambda =  (v^{(k)}_k)^{-1}$.

If we assume in addition that   $v^{(k)}\in (\smax^\vee)^n$. Then,
$\sval({\bf v})= \lambda  v^{(k)}=(v^{(k)}_k)^{-1} v^{(k)}$.
\end{proof}

Recall that by \Cref{theoremgeneric} or \Cref{corogeneric}, 
we have that generically in the moduli of the
 coefficients of the $\smax$-matrix $A$, 
all the eigenvalues $\gamma_k$ are simple and all the vectors 
 $v^{(k)}$ are in $(\smax^\vee)^n$, in which  case the above
result allows one to obtain the valuation
of all the eigenvalues and eigenvectors of $\bf A$.

\subsection{Numerical examples}
In the following, we illustrate the results of \Cref{th_asymptot} and \Cref{th_asymptot-vector} by some numerical examples. Computation of classical eigenvalues and eigenvectors is done by eig.m command in MATLAB 2019b
\begin{example}\label{ex_eig_n}
Consider
\begin{equation}
\label{defAex}
A= \begin{pmatrix}
5&4&3 &2& 1\\
4&4&3&2&1\\
3&3&3&2&1\\
2&2&2&2&1\\
1&1&1&1&1
\end{pmatrix} \in \pd_5(\smax)\quad \text{and}\quad {\bf A}(t)=\begin{pmatrix}
t^5&t^4&t^3 &t^2& t^1\\
t^4&t^4&t^3&t^2&t^1\\
t^3&t^3&t^3&t^2&t^1\\
t^2&t^2&t^2&t^2&t^1\\
t^1&t^1&t^1&t^1&t^1
\end{pmatrix} \enspace .\end{equation}
Seeing ${\bf A}$ as a matrix in the field of generalized (formal
or converging) Puiseux series,
we have $\sval({\bf A})=A$.
For any function $f:\R_+\to \R$, we define  the valuation  as
\[ \vall (f)=\limsup_{t\to\infty} \log|f(t)|/\log(t)\enspace .\]
If $f$ is already a converging Puiseux series with parameter $t$
(which is the case when $f$ is a polynomial), then
its valuation as a Puiseux series coincides with its
valuation as a function, and it can be approximated by 
\[ \vall_t(f)=\log|f(t)|/\log(t) \]
with $t$ large.
We can then define $\sval$ and $\sval_t$ accordingly, and apply
these maps entry-wise to $t$-parametric families of matrices and vectors.
So seeing now ${\bf A}$ as a family of  matrices parameterized by $t$, we have
\begin{equation}\label{def_sval}\sval({\bf A})=\sval_t({\bf A})= A \quad \forall t>0 \enspace.\end{equation}
  Denote $\gamma_i,\; i=1,\ldots,n=5$ the $\smax$-eigenvalues of $A$. Using \Cref{sym_eigs}, the $\gamma_i$ are the diagonal entries of $A$.
Denote also $\lambda_1(t)\geq \cdots \geq \lambda_n(t)$ the (classical) eigenvalues of ${\bf A}(t)$. 
By \Cref{th_asymptot}, we have $\sval(\lambda_i)=\gamma_i$, for all $i=1,\ldots,5$. In  \Cref{gamma}, we show the values of $\gamma_i$ and $\sval_t(\lambda_i)$
for  $t=10$ and $t=100$, which show the practical convergence
of $\sval_t(\lambda_i)$ towards $\gamma_i$. 
  
 \begin{table}[ht]
\caption{Values of $\gamma_i$ and $\sval_t(\lambda_i),\; i=1, \ldots, 5$,
for the matrices in \eqref{defAex}.}\label{gamma}
\begin{center}

 \begin{tabular}{c||c|c|c|c|c}

&$i=1$&$i=2$&$i=3$&$i=4$&$i=5$\\
\hline
\hline
$\gamma_i$&$5$&$4$&$3$&$2$&$1$\\

\hline

\hline
$\sval_t(\lambda_i) (t=10)$&5.0048&3.9543&2.9542&1.9542&0.9494\\
\hline

$\sval_t(\lambda_i) (t=100)$&5.0000&3.9978&2.9978&1.9978&0.9978\\
\end{tabular}
\end{center}
\end{table}
\end{example}

\begin{example}
Assume that $A$ and ${\bf A}(t)$ are as in \eqref{defAex}. 
Consider, for all $i=1,\ldots, 5$,
the  $\smax$-vector $v^{(i)}$ defined in  \Cref{vk}.
By \Cref{coro-simple-eigen}, it is a weak $\smax$-eigenvector of $A$ associated to the $\smax$-eigenvalue $\gamma_i$.
Denote, for all $i=1,\ldots, 5$,
 by ${\bf v}^{(i)}(t)$ the classical eigenvector of ${\bf A}(t)$
associated to the eigenvalue $\lambda_i(t)$, satisfying $({\bf v}^{(i)}(t))_i=1$.
Then, by \Cref{th_asymptot-vector}, we have 
$\sval({\bf v}^{(i)})\balance (v^{(i)}_i)^{-1} v^{(i)}$.
In  \Cref{vi1}, we show the vectors $(v^{(i)}_i)^{-1} v^{(i)},\;i=1,\dots,5$
(in the  first row)
and the vectors $\sval_t({\bf v}^{(i)}),\; i=1, \ldots, n$ for $t=10$ and $t=100$ (in second and  third row, respectively).
The results show the practical convergence of $\sval_t({\bf v}^{(i)})$
when $t$ goes to infinity, which holds because the vectors ${\bf v}^{(i)}(t)$
have a Puiseux series expansion in $t$.
The limit of the vectors $\sval_t({\bf v}^{(i)})$ satisfy
  $\lim_{t\to\infty} \sval_t({\bf v}^{(i)})=\sval({\bf v}^{(i)})\balance (v^{(i)}_i)^{-1} v^{(i)}$ as expected.
In particular, the $k$th entry of $\sval({\bf v}^{(i)})$
coincides with the $k$th entry of $(v^{(i)}_i)^{-1} v^{(i)}$
when the latter is in $\smax^\vee$.
Note that 
in this (nongeneric) example, for each $k$th entry of $(v^{(i)}_i)^{-1} v^{(i)}$
which is in $\smax^\circ$, the valuation of 
the  $k$th entry of ${\bf v}^{(i)}$ (that is $|(\sval({\bf v}^{(i)}))_k|$,
where the modulus is in the sense of $\smax$) is strictly lower than 
$|(v^{(i)}_i)^{-1} (v^{(i)})_k|$.

\begin{table}[ht]
\caption{Values of $(v^{(i)}_i)^{-1} v^{(i)}$ and
$\sval_t({\bf v}^{(i)}),$ 
$i=1, \ldots,5$, 
for the matrices in \eqref{defAex}.}\label{vi1}
\begin{center}
\footnotesize
  \begin{tabular}{c||c|c|c|c|c}
 &$i=1$&$i=2$&$i=3$&$i=4$&$i=5$\\
  \hline
  \hline
$(v^{(i)}_i)^{-1} v^{(i)}$&$\begin{pmatrix}
0\\
 -1\\
-2\\
-3\\
-4
\end{pmatrix}$&$\begin{pmatrix}
\ominus -1\\
 0\\
-1\\
-2\\
-3
\end{pmatrix}$&$\begin{pmatrix}
(-2)^{\circ}\\
\ominus -1\\
0\\
-1\\
-2
\end{pmatrix}$&$\begin{pmatrix}
(-3)^{\circ}\\
 (-2)^{\circ}\\
\ominus -1\\
0\\
-1
\end{pmatrix}$&$\begin{pmatrix}
(-4)^{\circ}\\
(-3)^{\circ}\\
(-2)^{\circ}\\
\ominus -1\\
0
\end{pmatrix}$\\


\hline
$\sval_t({\bf v}^{(i)})(t=10)$&$\begin{pmatrix}
0\\
 -0.9591\\
-1.9552\\
-2.9548\\
-3.9547
\end{pmatrix}$&$\begin{pmatrix}
\ominus-0.9542\\
 0\\
-0.9538\\
-1.9493\\
-2.9489
\end{pmatrix}$&$\begin{pmatrix}
-2.9450\\
\ominus -0.9493\\
0\\
-0.9538\\
-1.9494
\end{pmatrix}$&$\begin{pmatrix}
\ominus -5.9443\\
 -2.9447\\
\ominus -0.9494\\
0\\
-0.9542
\end{pmatrix}$&$\begin{pmatrix}
-9.9638\\
 \ominus -5.9591\\
-2.9548\\
\ominus -0.9547\\
0
\end{pmatrix}$\\
\hline

$\sval_t({\bf v}^{(i)})(t=100)$&$\begin{pmatrix}
0\\
 -0.9978\\
-1.9978\\
-2.9978\\
-3.9978
\end{pmatrix}$&$\begin{pmatrix}
\ominus-0.9978\\
 0\\
-0.9978\\
-1.9978\\
-2.9978
\end{pmatrix}$&$\begin{pmatrix}
-2.9978\\
\ominus -0.9978\\
0\\
-0.9978\\
-1.9978
\end{pmatrix}$&$\begin{pmatrix}
\ominus -5.9978\\
 -2.9978\\
\ominus -0.9978\\
0\\
-0.9978
\end{pmatrix}$&$\begin{pmatrix}
 -9.9979\\
 \ominus -5.9978\\
-2.9978\\
\ominus -0.9978\\
0
\end{pmatrix}$
\end{tabular}
\end{center}
\end{table}


 

\end{example}
\begin{example}\label{ex_eig_n2}
We generate a classical random positive definite matrix of size $100 \times 100$ by using the following MATLAB commands:
\begin{lstlisting}[style=Matlab-editor]
n = 100;
s = 15; rng(s);      
C1 = rand(n);
s = 12; rng(s);
C2 = rand(n);
C = C1-C2; % to have both negative and positive elements.
B = C*C'; % classical positive definite matrix
\end{lstlisting}
Consider now for any $t>0$, the map $\sval_t:\R\to\smax$ such that
$\sval_t(a)=\sign(a) \log(|a|)/ \log(t)$, for all $a\neq 0$ and
$\sval_t(0)=\zero$. This is as if $a$ is considered as the value in $t$ of
a function of $t$.
We can extend again $\sval_t$ to matrices and vectors.
We fix here $t=10$, and compute $A=\sval_t(B)$, for the
matrix $B$ randomly generated as above.
Then, generically so almost surely, we have that
$A$ is a positive definite $\smax$-matrix
and satisfies the conditions of \Cref{th_asymptot-vector}.
Let $\gamma_i$, $i=1,\ldots,n$ be the $\smax$-eigenvalues of $A$ and
$\lambda_i$, $i=1,\ldots, n$ be the eigenvalues of $B$.
In \Cref{rel_error_eigval}, we computed the relative error $\frac{|\sval_t(\lambda_i)-\gamma_i|}{\sval_t(\lambda_i)}$ for $i=1,\ldots, 100$. 
One can see that the errors are less than $1.6\times 10^{-13}$.
 \begin{figure}[h!]
\centerline{\includegraphics[width=5in]{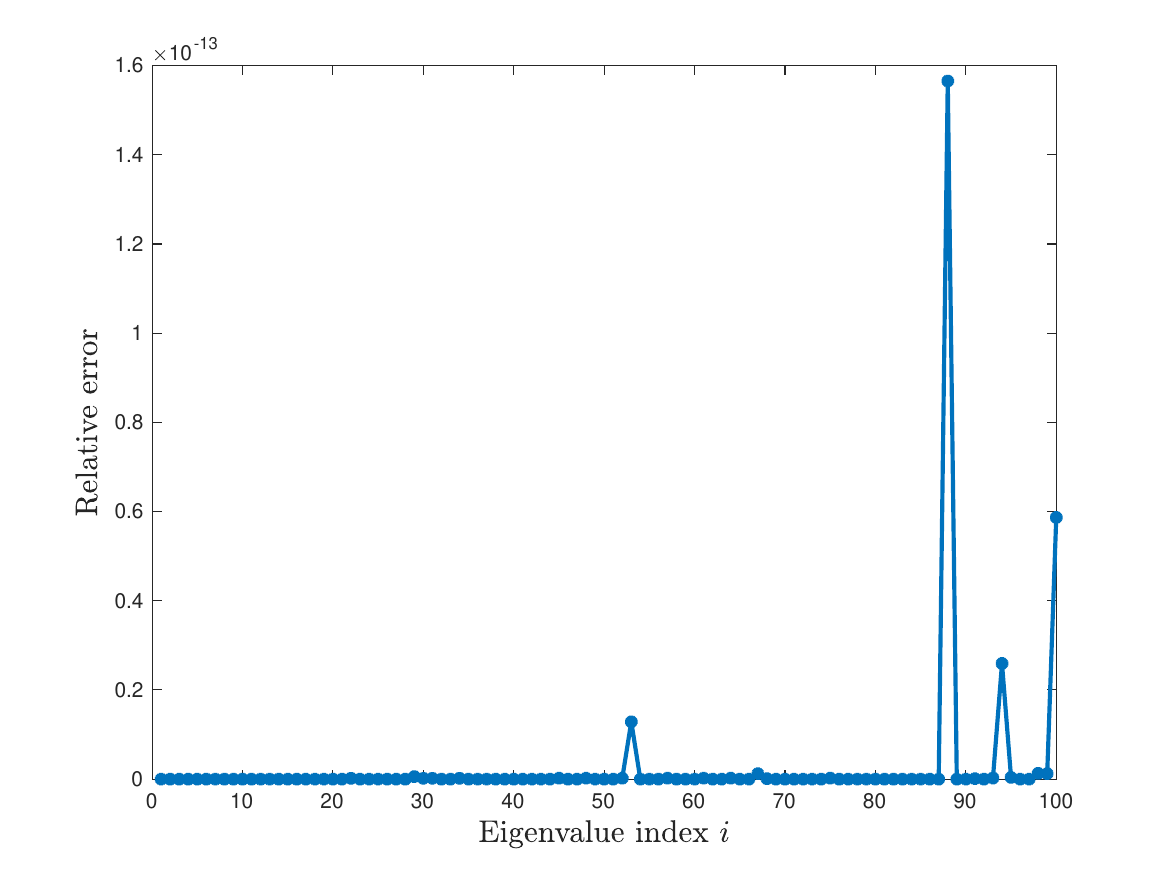}}
\caption{ The values $\frac{|\sval_t(\lambda_i)-\gamma_i|}{\sval_t(\lambda_i)}$ for $i=1,\ldots, 100$ and $t=10$ in \Cref{ex_eig_n2} }\label{rel_error_eigval}
\end{figure}
\end{example}

\bibliographystyle{alpha} 
\bibliography{refs} 

\appendix
\section{Detailed proofs of some results}

Let us first prove the following result which corresponds to the statement of \Cref{def_pd1} for $n=2$.
\begin{lemma}\label{lemma_sergey_pd}
 Let $a, b, c \in \smax^{\vee}$. Then, 
\[\zero \lsign (a x_1^{ 2}) \oplus (b  x_1  x_2) \oplus (c  x_2^{ 2})\quad \forall (x_1, x_2) \in(\smax^{\vee})^2\setminus \{(\zero,\zero)\}\]
if and only if
\[\zero \lsign a,\; \zero \lsign c, \;b^{ 2} \lsign a c\enspace .\]
\end{lemma}
\begin{proof}
$(\Rightarrow)$ Considering $x_1 = \unit $ and $x_2=\zero$, we get $\zero \lsign a$. Similarly, by considering $x_1 = \zero $ and $x_2=\unit$, we have $\zero \lsign c$. This shows that  $a,c \in \smax^{\oplus}\setminus \zero$. Then taking $x_1=a^{ -\frac{1}{2}}$ and $x_2=\eta  c^{ -\frac{1}{2}}$ with $\eta \in \{\ominus \unit, \unit\}$, we get 
\[\zero \lsign \unit \oplus (\eta  b  a^{ -\frac{1}{2}}  c^{ -\frac{1}{2}})\oplus \unit =\unit \oplus (\eta  b  a^{ -\frac{1}{2}}  c^{ -\frac{1}{2}}),\]
which is possible only if $|b|  a^{ -\frac{1}{2}}  c^{ -\frac{1}{2}} \lsign \unit$ or equivalently if $b^{ 2}  \lsign a  c$,
using second part of \Cref{product_order} and \Cref{modulus_order}.

$(\Leftarrow)$ 
Assume $a,c \in \smax^{\oplus} \setminus \{\zero\}$
and $b\in \smax^\vee$ are such that $b^{ 2} \lsign a c$. By the change of variable $x_1= y_1  a^{ -\frac{1}{2}}$ and $x_2= y_2  c^{ -\frac{1}{2}}$,  we have 
$\zero \lsign (a x_1^{ 2}) \oplus (b  x_1  x_2) \oplus (c  x_2^{ 2})$
iff
\begin{equation}\label{iff}
 \zero \lsign (y_1^{ 2}) \oplus (u  y_1  y_2) \oplus (y_2^{ 2}),
\end{equation}
where $u = b  a^{ -\frac{1}{2}}  c^{ -\frac{1}{2}}$ is such that $|u|\lsign \unit$ (again by \Cref{product_order} and \Cref{modulus_order}).
W.l.o.g. we may assume that $|y_2| \leqsign |y_1|$.
This implies that $y_1\neq \zero$, since $(x_1,x_2)\neq (\zero,\zero)$,
and that $y_2^{ 2} \leqsign y_1^{ 2}$, by  \Cref{modulus_order}.  Then 
$|u  y_1  y_2|\leqsign |u|  |y_1|^{ 2} 
\lsign  |y_1|^{ 2} =|y_1^{ 2}|$ (again by \Cref{product_order}).
So since $\leqsign$ and $\prec$ coincide in $\smax^\oplus$,
we get $|u  y_1  y_2|\prec |y_1^{ 2}|$.
Using \Cref{property-preceq}, we obtain that 
the right hand side of \eqref{iff} is equal to $y_1^{ 2} \gsign \zero$.
\end{proof}

\begin{proof}[Proof of \Cref{def_pd1}]
Let $S$ be the set given in \Cref{def_pd1}, and let $A \in S$.
We need to prove that $A$ is a  $\pd$ matrix.
Let $x\in (\smax^\vee)^n\setminus \{\zero\}$, and $i,j=1,\ldots, n$ with $i\neq j$.
Applying the ``if'' part of \Cref{lemma_sergey_pd} to $a=a_{ii}$, $b=a_{ij}$ and $c=a_{jj}$,
 we have
$\zero  \lsign (x_i a_{ii}  x_i ) \oplus (x_i  a_{ij}  x_j) \oplus (x_j  a_{jj}  x_j)$, if $(x_i,x_j)\neq (\zero,\zero)$.
Moreover the previous expression is equal to $\zero$ if $(x_i,x_j)= (\zero,\zero)$.
By summing these inequalities over all $i\neq j$, and
using that $x$ is not the $\zero$ vector, and idempotency of addition,
 we obtain
 $\zero \lsign x^T A  x$. This shows that
$S\subseteq \pd_{n}(\smax^\vee)$.
Let us prove the reverse inclusion.
Consider $A \in \pd_n(\smax^{\vee})$. Then, every $2 \times 2$ leading submatrix of $A$  is   $\pd$ and therefore by applying the ``only if'' part of \Cref{lemma_sergey_pd}, we obtain, for $i\neq j$, $\zero \lsign a_{ii}$, $\zero \lsign a_{jj}$ and $a_{ij}^{ 2} \lsign a_{ii} a_{jj}$.
Hence, $A\in S$, which shows that $\pd_{n}(\smax^\vee)\subseteq S$ and
concludes the proof of \Cref{def_pd1}.

\end{proof}
\end{document}